\documentclass[lettersize,journal]{IEEEtran}
\usepackage{amsmath,amsfonts, bm}
\usepackage{array}
\usepackage{listings}
\usepackage{adjustbox}
\usepackage[group-digits=true,group-minimum-digits=4]{siunitx}
\usepackage{tikz,pgfplots}
\usepackage{textcomp}
\usepackage{stfloats}
\usepackage{url}
\usepackage{verbatim}

\usepackage{graphicx} 
\usepackage{float}
\usepackage{subcaption}

\usepackage{acronym}
\usepackage{caption}
\usepackage{comment}
\usepackage{soul}
\usepackage{color}
\usepackage{mathtools}
\usepackage{balance}
\usepackage{multicol}
\usepackage{longtable}
\usepackage{amssymb}
\usepackage{booktabs, makecell, tabularx}
\usepackage{longtable}
\usepackage[section]{placeins}
\usepackage{nomencl}
\usepackage{footmisc}
\usepackage{cite}
\usepackage{xcolor}
\usepackage{caption}
\usepackage{setspace}
\usepackage{amsthm}
\usepackage{academicons}
\usepackage[acronym]{glossaries}
\usepackage{mathtools, cuted}
\usepackage{lipsum}
\usepackage{balance}
\usepackage[font={footnotesize},labelfont={footnotesize}]{caption}
\usepackage{xpatch}
\theoremstyle{plain}
\newtheorem{theorem}{Theorem}    
\newtheorem{corollary}{Corollary}[theorem]

\usepackage[colorlinks,
            urlcolor=blue,
            linkcolor=blue,
            anchorcolor=blue,
            citecolor=blue
           ]{hyperref}

\makeatletter
\xpatchcmd{\proof}{\@addpunct{.}}{\@addpunct{:}}{}{}
\makeatother

\pgfplotsset{compat=1.18}
\begin{document}

\title{Optical Intelligent Reflecting Surfaces Empowering Non-Terrestrial Communications }

\author{Shunyuan Shang,~Emna~Zedini,~\IEEEmembership{Member,~IEEE,}~Abla Kammoun,~\IEEEmembership{Member,~IEEE,}~and~Mohamed-Slim~Alouini,~\IEEEmembership{Fellow,~IEEE} 
\thanks{(\textit{Corresponding author: Shunyuan Shang})}
\thanks{S. Shang, A. Kammoun and M.-S. Alouini are with the Computer, Electrical, and Mathematical Science and Engineering  Division, King Abdullah University of Science and Technology, Thuwal, Makkah Province, Saudi Arabia (e-mail: shunyuan.shang@kaust.edu.sa;abla.kammoun@kaust.edu.sa; slim.alouini@kaust.edu.sa).}
    \thanks{E. Zedini is with the College of Innovation and Technology, University of Michigan-Flint, Flint, MI, USA (e-mail: ezedini@umich.edu).}}


\maketitle

\begin{abstract}
In this work, we propose an innovative system that combines high-altitude platforms (HAPs) and optical intelligent reflecting surfaces (OIRS) to address line-of-sight (LOS) challenges in urban environments. Our three-hops system setup includes an optical ground station (OGS), a HAP, an OIRS, and a user. Signals are transmitted from the OGS to the HAP via a free space optical (FSO) link, with the HAP functioning as an amplify-and-forward (AF) relay that redirects signals through an OIRS, effectively bypassing obstacles such as buildings and trees to improve connectivity for non-line-of-sight (NLOS) User. For the OIRS link, we address key channel impairments, including atmospheric turbulence, pointing errors, attenuation, and geometric and misalignment losses (GML). An accurate approximation for the Hoyt-distributed GML model is derived, enabling us to obtain closed-form expressions for outage probability (OP) and various performance metrics, such as average bit error rate (BER) and channel capacity of the OIRS-assisted FSO link. Furthermore, we analyze the end-to-end signal-to-noise ratio (SNR) and derive closed-form expressions for OP and performance metrics. Asymptotic expressions are provided for high-SNR regimes, allowing the system's diversity order to be calculated.

\end{abstract}


\IEEEpeerreviewmaketitle

\section{Introduction}
As the demand for ultra-fast data connectivity continues to grow with the proliferation of digital technologies and connected devices \cite{dang2020should,qadir2023towards}, FSO communication is gaining prominence as a critical enabler in the evolution of communication technologies beyond 5G towards 6G, particularly in point-to-point networks. Its advantages over traditional radio frequency (RF) communication include significantly broader bandwidth, higher channel capacity, and cost-efficiency due to operation in an unlicensed spectrum \cite{sood2018analysis,al2020survey}. Additionally, FSO communication offers enhanced immunity to interference and robust security. FSO transceivers are also more affordable and more accessible to deploy compared to their RF counterparts \cite{khalighi2014survey}. However, FSO systems rely on a LOS connection between the transmitter and the receiver and face challenges such as fading due to atmospheric turbulence, significant atmospheric losses during dense fog and heavy snowfall. Although numerous techniques have been suggested to mitigate these impairments, such as aperture averaging \cite{khalighi2009fading}, diversity methods\cite{priyadarshani2023earth}, and adaptive optics \cite{ata2022haps}, the need for a clear LOS link continues to limit the applicability of FSO systems significantly. 

Recent research has extensively explored the integration of FSO systems with non-terrestrial networks (NTNs) to mitigate the limitations imposed by the LOS requirements of FSO communication \cite{turk2024design,singya2022mixed,samy2024enabling,elamassie2024multi,shah2021adaptive,priyadarshani2023earth,yahia2022haps,deng2025distributedcoordinationheterogeneousnonterrestrial}.  High-altitude platforms (HAPs) are particularly advantageous for enhancing FSO communication among the various non-terrestrial solutions. In \cite{turk2024design, singya2022mixed,samy2024enabling}, the authors leverage multiple HAPs to enable optical communication links, successfully achieving data transmission from earth stations to ground User and from satellites to ground User. 

Nevertheless, when User is located in urban areas, communication between HAPs and User can still be obstructed by tall buildings and trees. To address this issue, the authors in \cite{OIRS1} propose utilizing OIRS to mitigate LOS requirements for FSO communication systems. OIRS are energy- and cost-efficient since they consist of passive elements and can be deployed on existing infrastructure, such as buildings, effectively addressing the LOS dependency in FSO communication. The authors develop statistical channel models for both two-dimensional (2D) and three-dimensional (3D) systems.  
{
The key difference between the 2D and 3D system models lies in their treatment of spatial beam deviation and incident–reflection geometry. In the 2D model, the incident and reflected optical beams are assumed to lie within the same plane, and the geometric misalignment loss (GML) is evaluated based solely on beam jitters along two orthogonal in-plane axes. In contrast, the 3D model accommodates arbitrary orientations of the incident and reflected beams, allowing the reflection plane to differ from the incidence plane. Consequently, the 3D GML formulation captures the combined spatial jitters of the laser source, OIRS, and lens along all three spatial dimensions, offering a more comprehensive characterization of misalignment effects in practical scenarios. It is also worth noting that the 3D model generalizes the 2D case: by selecting specific system parameters, the 3D GML expression naturally reduces to the 2D form as a special case.
}

There is now extensive research on OIRS \cite{pang2022optical,nguyen2022design,naik2022evaluation,ndjiongue2021analysis,malik2022performance}. The authors in \cite{pang2022optical}  established an end-to-end (e2e) three hops model for a hybrid FSO/RF channel, where a more general 3D model is employed for the OIRS channel, incorporating a decode-and-forward (DF) relay. The authors in \cite{nguyen2022design}  developed an e2e dual-hop model for an FSO channel, utilizing a 2D model for the OIRS channel with an AF relay. In \cite{naik2022evaluation,ndjiongue2021analysis,malik2022performance}, the authors analyzed the application of FSO models incorporating OIRS in building-to-building (B2B) communications and underwater communication scenarios. {
Both amplify-and-forward (AF) and decode-and-forward (DF) relaying schemes are widely used in FSO systems. In this work, we adopt a fixed-gain AF relay at the HAP. The rationale behind this choice is summarized as follows:
\begin{itemize}
\item {Advantages of AF over DF:}
    AF relaying avoids the need for full signal decoding and re-encoding, which significantly reduces processing complexity. It is better suited for power- and resource-constrained relay platforms, such as UAVs or HAPs, where hardware simplicity and low latency are crucial. 
\item {Disadvantages of AF:}
AF relays amplify not only the desired signal but also the accumulated noise, potentially degrading SNR at the destination. Also, the end-to-end SNR becomes a nonlinear function of the individual hop SNRs, which complicates statistical analysis and closed-form performance evaluation.
\end{itemize}
Despite the analytical complexity introduced by AF relaying, it offers a practical and lightweight solution for non-terrestrial FSO deployments. 
}
 \begin{table*}[b]
\centering
{\caption{{ Comparison of related works involving OIRS-Assisted FSO Systems}}
\label{tab:OIRSComparison}
\begin{tabular}{|c|c|c|c|c|c|c|}
\hline
{Reference} & {System Type} & {Hops} & {OIRS Model} & {Relay Type} & Scenario \\ 
\hline
\cite{OIRS1} & Single FSO link & 2 & 2D/3D  & None  & None \\ \hline \cite{pang2022optical} & Hybrid FSO/RF & 3 & 3D & DF  &B2B \\
\hline
\cite{nguyen2022design} & Hybrid FSO/RF & 3 & 2D & AF &NTN \\
\hline
\cite{naik2022evaluation} & Single FSO link & 2 & None  & None   &Underwater\\
\hline
\cite{ndjiongue2021analysis} & Single FSO link & 2 & 3D & None &NTN\\
\hline
\cite{malik2022performance} & Single FSO link & 2 & 3D  & None  &NTN\\
\hline
\cite{ndjiongue2021analysis} & Single FSO link & 2 & 3D & None  &NTN\\
\hline
\cite{shang2025novelhybridopticalstar} & Dual FSO-RF link & 3 & 3D & AF  &NTN\\
\hline
Our work & Dual FSO-FSO links & 3 & 3D  & AF  &NTN \\
\hline
\end{tabular}}
\end{table*}

{ To the best of the authors' knowledge, this is the first study to incorporate a  3D OIRS model and fixed gain AF relay into a three-hop system. Table~\ref{tab:OIRSComparison} provides a comparative summary of representative existing works on OIRS-assisted FSO systems and highlights the key distinctions of our proposed model.}  As shown in Fig.~\ref{fig1}, the proposed system consists of four interconnected nodes: the optical ground station (OGS), HAP, OIRS, and the User. { Compared to terrestrial OIRS-assisted FSO systems, the HAP offers notable advantages. The HAP acts as an aerial relay with a quasi-omnidirectional field of view, ensuring a reliable LOS link from the OGS regardless of urban obstructions. Additionally, the HAP’s altitude and mobility enable adaptive deployment, allowing coverage extension or real-time reconfiguration in response to changing network demands or environmental conditions.}

{ Nevertheless, despite the elevated altitude and extensive coverage capabilities offered by the HAP, the direct FSO link between the HAP and the User may still experience significant obstructions in dense urban environments. Such obstructions typically include high-rise buildings, dense vegetation, large billboards, and other urban infrastructure, all of which can severely attenuate or completely block optical signals. As illustrated by the dashed blue line in Fig.~\ref{fig1}, these obstacles substantially degrade the reliability of the communication link. So, relying solely on HAPs to maintain FSO coverage for all urban users would require deploying multiple HAPs within a small area, which significantly increases energy consumption and operational cost. In contrast, equipping building rooftops with OIRS elements offers a low-power, scalable, and infrastructure-compatible alternative that efficiently overcomes non-line-of-sight (NLOS) blockages. Placing the OIRS on rooftops also enables the use of rooftop solar panels for power supply, which can further reduce the system cost compared to deploying the OIRS on HAPs. }
\begin{figure}[!ht]
\centering\includegraphics[width=0.35\textwidth]{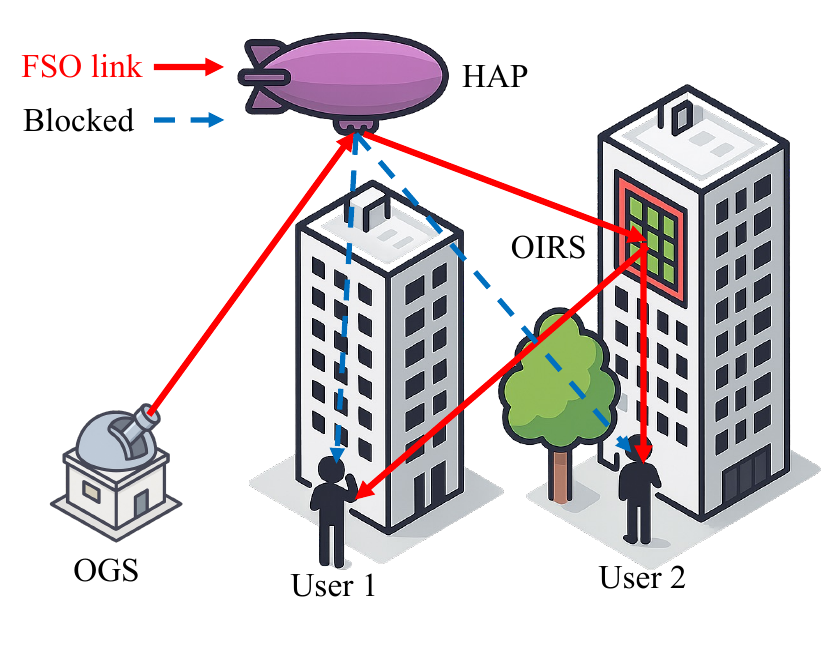}
    \caption{HAP-Ground Integrated Communication System with Optical Intelligent Reflecting Surfaces.}
    \label{fig1}
\end{figure}

{ 
Despite the considerable potential of the proposed system, deriving a closed-form expression for the e2e SNR CDF poses significant analytical challenges due to the following reasons:
\begin{itemize}
\item{Lack of Closed-Form Composite PDF under GG Turbulence and 3D GML:} In \cite{OIRS1}, there is no available closed-form expression for the composite channel's PDF when both GG atmospheric turbulence and 3D GML are jointly considered.

\item{Meijer-$G$ Functions in the First-Hop SNR Distribution:} The statistical characterization of the first FSO hop, namely from the OGS to the HAP, inherently involves the Meijer-$G$ function in the expression of the SNR’s PDF due to the GG turbulence and pointing error effects. While the Meijer-$G$ function provides a powerful tool for representing complex integrals, it also complicates the derivation of closed-form expressions, especially when combined with other non-elementary functions.

\item The use of a fixed-gain AF relay further exacerbates the difficulty, as the e2e SNR expression becomes a nonlinear function of the individual link SNRs. In contrast to DF relaying, where the end-to-end CDF can often be obtained by simply applying a product transform of the individual link CDFs, AF relaying requires computing an integral over the joint distribution of the two hops.
\end{itemize}

Due to the above challenges, this work proposes a novel approximation framework to simplify the composite channel model and obtain accurate closed-form and asymptotic results using advanced special functions, including the bivariate Fox-H function.} 
{ The main contributions of this work are as follows:
\begin{itemize}
\item For the 3D OIRS model incorporating GML following a Hoyt distribution and atmospheric turbulence modeled by the G-G distribution as proposed in \cite{OIRS1}, the authors did not provide a closed-form expression for the composite channel. In this work, we address this limitation by deriving an accurate and analytically tractable approximate expression, which serves as a foundation for further performance analysis.

\item Based on our derived approximate expression for the 3D OIRS model incorporating geometric misalignment loss (GML) following a Hoyt distribution and atmospheric turbulence modeled by the Gamma-Gamma (GG) distribution, we obtain closed-form expressions for the probability density function (PDF) and cumulative distribution function (CDF) of the SNR for the HAP–OIRS–User link. Leveraging the derived CDF, we further obtain closed-form expressions for key performance metrics of this link, including OP, average bit error rate (BER), channel capacity, and the statistical moments of the SNR.

\item Based on our derived approximate expression for the 3D OIRS model incorporating GML following a Hoyt distribution and atmospheric turbulence modeled by the G-G distribution, we successfully obtain closed-form expressions for the PDF and CDF  of the SNR for the complete end-to-end link, spanning from the OGS to the HAP, then to the OIRS, and finally to the User. And then  utilizing the closed-form expression for the CDF of the SNR for the end-to-end link from the HAP to the User, we derive analytical expressions for several key performance metrics. These include the OP, average BER for On-Off Keying (OOK), M-ary Quadrature Amplitude Modulation (M-QAM), and M-ary Phase Shift Keying (M-PSK), as well as the channel capacity and statistical moments of SNR for the end-to-end link.

\item Considering the complexity of the bivariate Fox-H function, we also provide asymptotic results for OP and average BER of the end to end link from to OGS to the User. From the asymptotic result of the OP, we derive the system's diversity order.
\end{itemize}}

The remainder of this work is structured as follows: Section II presents the system and channel models. In Section III, we perform a statistical analysis of the SNR for the OIRS link, deriving analytical expressions for various performance metrics. We also analyze the end-to-end SNR, providing analytical expressions and asymptotic results in the high-SNR regime. Section IV presents numerical and simulation results, and concluding remarks are provided in Section V.

\section{ System models for two single links }
To clarify the positional relationships, Fig.~\ref{fig2} illustrates the relative 3D positions of the OGS, HAP, OIRS, and User.  The OGS is located at $(0, 0, H_O)$, and it transmits signals to the HAP through a FSO link with a zenith angle of $\zeta_1$. The HAP is positioned at $(X_H, Y_H, H_H)$, and the distance between the OGS and the HAP is denoted by $d_{OH}$. As an AF relay, the HAP amplifies the signal received from the OGS and forwards it to the User or via the IRS.  The distance between the HAP and the IRS is $d_{HI}$, and the corresponding zenith angle is $\zeta_2$. At the IRS, the signal is reflected and then transmitted to the User located at $(X_U, Y_U, H_U)$. The distance between the IRS and the User are $d_{IU}$, with a zenith angle of $\zeta_3$. As shown in Fig.~\ref{fig3}, a new coordinate system is established.   The incident angle \( \theta_i \) is the angle between the incoming beam and the surface normal, while the reflection angle \( \theta_r \) is the angle between the reflected beam and the normal. Unlike conventional reflection, the IRS intelligently adjusts \( \theta_r \) to optimize signal direction. The angle \( \theta_{rl} \) is the angle between the reflected beam and the receiving lens, while the azimuthal angle \( \phi_r \) defines the reflected beam's direction in the horizontal \( xy \)-plane. {
While this work focuses on a single-User scenario to enable a tractable analytical framework, the proposed system architecture is inherently scalable to multi-user deployments. In practical urban environments, multiple OIRS units can be installed on the rooftops of different buildings to serve users located in distinct NLOS regions. }

\begin{figure}[!ht]
\centering\includegraphics[width=0.4\textwidth]{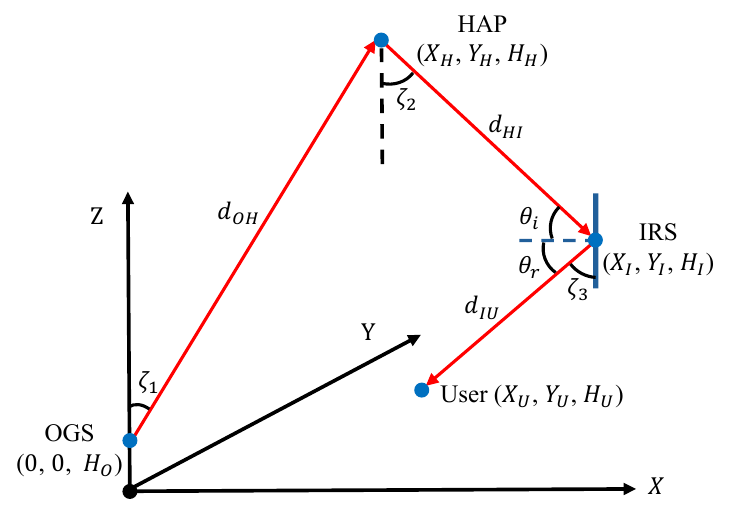}
    \caption{Precise Positioning of the Integrated HAP-Ground Communication System with OIRS.}
    \label{fig2}
\end{figure}
\begin{figure}[!ht]
\centering\includegraphics[width=0.35\textwidth]{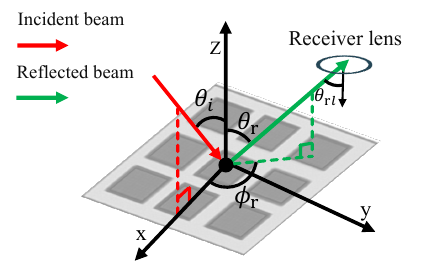}
    \caption{Schematic diagram of the OIRS System.}
    \label{fig3}
\end{figure}

\subsection{System model from OGS to HAP}
  
\subsubsection{Channel Model}
{
To accurately model the FSO communication channel gain $h_1$ from the OGS to the HAP, it is essential to consider the combined effects of attenuation losses \( h_{p1} \), atmospheric turbulence \( h_{a1} \) and pointing errors \( h_{g1} \). The channel model \( h_1 \) can be given as
\begin{align}
    h_1=h_{p1}h_{a1}h_{g1}.
\end{align}

Atmospheric attenuation resulting from absorption and scattering effects is given by the Beer-Lambert law, as detailed in \cite[Eq.~(1)]{hal1}. This law describes how the intensity of light decreases exponentially with the distance it travels through the atmosphere, affecting the overall signal strength, which can be expressed as
\begin{align}
    h_{p1}=\exp\left[-C_h (\lambda)d_{OH}\right],
\end{align}
where $\lambda$ denotes the wavelength in nanometers [nm]. $C_h (\lambda)$ represents the attenuation coefﬁcient which is specified in \cite[Eq.~(4)]{hal1} as
\begin{align}
    C_h(\lambda) = \frac{3.912}{V}\left(\frac{\lambda}{550}\right)^{-q_V},
\end{align}
where  $V$ represents the visibility in kilometers [$\mathrm{Km}$], and $q_V$ is the corresponding coefficient.

Since the distance between the OGS and the HAP is tens of kilometers in this work, we choose to model the atmospheric turbulence from OGS to HAP using the G-G distribution, which is provided as \cite[Eq.(56), pp.462]{Laserbeam}
\begin{align}
\label{pdfhat}
    f_{h_{a 1}}(h_{a 1})=\frac{2(\alpha_1 \beta_1)^{\frac{\alpha_1+\beta_1}{2}}h_{a t}^{\frac{\alpha_1+\beta_1}{2}-1}}{\Gamma(\alpha_1) \Gamma(\beta_1)}  K_{\alpha_1-\beta_1}\!\left(2 \sqrt{\alpha_1 \beta_1 h_{a 1}}\right)\!,
\end{align}
where $K_a (\cdot)$ indicates the modified Bessel function of the second kind with the order $a$, and $\Gamma (\cdot)$ represents the Gamma function,
$\alpha_1=1/[\exp(\sigma_{\ln X1}^2)-1]$, and $\beta_1=1/[\exp(\sigma_{\ln Y1}^2 )-1]$. The large-scale log variance $\sigma_{\ln X1}^2$ and the small-scale log variance $\sigma_{\ln Y1}^2$ are provided by \cite[Eqs. (97) and (101), pp.352]{Laserbeam} as
\begin{align}
    \sigma_{\ln X1}^2=\frac{0.49\sigma_{B1}^2}{[1+0.56(1+\Theta) \sigma_{B1}^{12/5} ]^{7/6}},
\end{align}
\begin{align}
    \sigma_{\ln Y1}^2=\frac{0.51\sigma_{B1}^2}{[1+0.69\sigma_{B1}^{12/5} ]^{5/6}},
\end{align} 
where $\sigma_{B1}^2$ represents the Rytov variance, quantifying the scintillation index in weak turbulence scenarios. For the uplink propagation case, $\sigma_{B1}^2$ is expressed as \cite[Eq.~(12)]{ata2022haps}
\begin{equation}
\begin{aligned}
    &\sigma_{B1}^{2} =  \,\,8.7 {k_w}^{7 / 6}\left(H_H-H_O\right)^{5 / 6} \sec ^{11 / 6}(\zeta_1)\operatorname{Re}\Big(\int_{H_O}^{H_H} \\ &\times C_{n}^{2}(l)\Big\{\left[\Lambda {\xi_1}^{2}+\mathrm{i} \xi_1(1-\overline{\Theta} \xi_1)\right]^{5 / 6}-\Lambda^{5 / 6} {\xi_1}^{5 / 3}\Big\} \mathrm{d} l\Big),
\end{aligned}
\end{equation}
where \( k_w = 2\pi/\lambda \) represents the wave number, with \( \lambda \) being the wavelength measured in meters [$\mathrm{m}$]. The function \( \sec(\cdot) \) denotes the secant function. The Fresnel ratio of the Gaussian beam at the receiver is expressed as \( \Lambda = \Lambda_0/(\Lambda_0^2 + \Theta_0^2) \), where \( \Lambda_0 = 2d_{OH}/(k_w \omega_{01}^2) \) and \( \omega_{01} \) is the initial beam radius from OGS. The beam curvature parameter at the transmitter is given by \( \Theta_0 = 1 - d_{OH}/F_0 \). For the uplink propagation case, the normalized distance parameter is \( \xi_1 = (l - H_H)/(H_O - H_H) \). The complementary parameter is defined as \( \overline{\Theta} = 1 - \Theta \), with the beam curvature parameter at the receiver being \( \Theta = \Theta_0/(\Theta_0^2 + \Lambda_0^2) \). The turbulence structure constant \( C_n^2(l) \) follows the conventionally used Hufnagel-Valley (HV) model, which is specified by \cite[Eq.~(1), pp.481]{Laserbeam}
\begin{equation}
\begin{aligned}
         C_{n}^{2}(l)&\,\,=0.00594(\omega / 27)^{2}\left(10^{-5} l\right)^{10}\exp (-l / 1000)\\ 
 &+2.7 \times 10^{-16} \exp (-l / 1500)+A \exp (-l / 1000),
\end{aligned}
\end{equation}
where \( l \) is measured in meters [$\mathrm{m}$], \( \omega \) denotes the root mean square (RMS) wind speed in meters per second [$\mathrm{m/s}$], and \( A \) represents the nominal value of \( C_{n}^{2}(0) \), as detailed in \cite[pp. 481]{Laserbeam}.

Pointing errors \( h_{g1} \) are a critical impairment in FSO systems, especially over long-distance. These errors occur due to various dynamic and mechanical disturbances that cause deviations in the beam's intended trajectory. In our system, potential sources of pointing error include HAP platform instability, which arises from residual vibrations and orientation jitter caused by wind gusts, atmospheric turbulence, and limitations in onboard stabilization mechanisms. Additionally, mechanical inaccuracies at the OGS, such as tracking errors, thermal deformations, and structural vibrations in optical mounts, can introduce misalignments.  The PDF of \( h_{g1} \) is given as \cite[Eq.~(11)]{farid2007outage}
\begin{align}
\label{pdfhpl}
    f_{g1}(h_{g1}) = \frac{\eta_s^2}{A_0} \left( \frac{h_{p1}}{A_0} \right)^{\eta_s^2  - 1},
\end{align}
where $\eta_{s}$ is defined as $\omega_{b}\sqrt{\sqrt{\pi A_{01}}/\left[2 v_e\exp \left(-v_e^{2}\right)\right]} /\left(2 \sigma_{S0}\right)$, $v_e=r_{a}\sqrt{\pi / 2}  / \omega_{b}$ with $\omega_b$ representing the beam waist, $A_{01}=\text{erf}^2(v_e)$, and ${\rm erf}(\cdot)$ denoting the error function. \( r_a \) represents the radius of the receiver aperture.  

\subsubsection{ SNR Statistics from OGS to HAP}
Let \( s_1(t) \) be the signal transmitted by the OGS. The received signal by HAP can be written as
\begin{align}
    y_{H1}(t)=\sqrt{\eta_1 P_{01} }h_1 s_1(t) + n_{H}(t),
\end{align}
where \( P_{01} \) is the transmit power from the HAP, and \( \eta_1 \) is the optical-to-electrical conversion factor. The noise term \( n_{H}(t) \sim \mathcal{N}\left(0, \sigma_{H}^{2}\right) \) represents Gaussian noise. Therefore, SNR of the FSO link \( \gamma_{H1} \) can be formulated considering both IM/DD and heterodyne detections as
\begin{align}
\label{gammaHoverline}
    \gamma_{H1}=\overline{\gamma}_{H1} h^{r_1},
\end{align}
where  \( \overline{\gamma}_{H1} \) is defined as  \( \overline{\gamma}_{H1} = \frac{\left(\eta_1 P_{01}\right)^{r_1/2}}{\sigma_{H}^{2}} \), where  \( r_1 = 1 \) corresponds to the heterodyne detection technique, and \( r_1 = 2 \) corresponds to IM/DD.} The  CDF of $\gamma_{H1}$ can be expressed as \cite[Eq.~(17)]{shang_enhancing}
\begin{align}
\label{CDFH}
\nonumber F_{\gamma_{H1}}&\left(\gamma_{H1}\right) =1-\frac{\eta_{s}^{2} }{\Gamma(\alpha_1) \Gamma(\beta_1)}\\ 
&   \times{\rm G}_{2,4}^{4,0}\left[\frac{\alpha_1 \beta_1}{A_{0} h_{p1} }\left(\frac{ \gamma_{H1}}{\overline{\gamma}_{H1}}\right)^{\frac{1}{r_1}} \Bigg| \begin{array}{c}
1+\eta_{s}^{2} , 1 \\
0, \eta_{s}^{2} , \alpha_1, \beta_1
\end{array}\right] . 
\end{align}

\subsection{System model from HAP to User}
\subsubsection{Channel Model}
To model the FSO communication channel gain \( h_2 \) for the link from HAP to User, we consider the combined effects of attenuation losses \( h_{p2} \), atmospheric turbulence \( h_{a2} \), and GML \( h_{g2} \). Incorporating these factors, the channel gain \( h_2 \) can be expressed as
\begin{equation}
    h_2 = h_{p2}  h_{a2}  h_{g2}.
\end{equation}

Let \( \zeta_p \) denote the reflection efficiency.  The absorption occurring at the IRS can be considered part of the overall atmospheric loss, modeled by \( h_{p2} \), which is expressed as \cite[Eq.(7)]{OIRS1}
\begin{equation}
    h_{p2} =  \zeta_p 10^{-\kappa (d_{HI}+d_{IU})/10},
\end{equation}
where \( \kappa \) is the absorption coefficient.  { Due to non-specular reflection, phase deviations may arise at the OIRS, as characterized in \cite[Eq.~(28)]{OIRS1}. However, in this work, we assume perfect phase control at the OIRS, i.e., the induced phase shifts are ideal and free from such deviations.} We utilize the 3D GML model  described in \cite{OIRS1}, which is given as \cite[Eq.~(32)]{OIRS1}
\begin{equation}
\label{PDFhg2}
\begin{aligned}
        f_{h_{g2}}(h_{g2}) = &\frac{\varpi}{A_{02}} \left( \frac{h_{g2}}{A_{02}} \right)^{\frac{(1 + q_g^2)\varpi}{2q_g} - 1} \\&\times I_0 \left( -\frac{(1 - q_g^2)\varpi}{2q_g} \ln \left( \frac{h_{g2}}{A_{02}} \right) \right),  0 \leq h_g \leq A_{02}.
\end{aligned}
\end{equation}
we assume that \( \phi_r = \pi \) and \( \theta_{rl} = 0 \). Under these conditions, the parameters in the formula can be expressed as: \( \varpi = \frac{(1 + q_g^2)t_g}{4q_g\Omega} \), $t_g = \frac{\pi a_l^2}{4 \nu_1 \nu_2} \sqrt{\frac{\pi  \text{erf}(\nu_1)\text{erf}(\nu_2)}{\nu_1 \nu_2 \exp\left( -(\nu_1^2 + \nu_2^2) \right)}}$ \( A_{02} = \text{erf}(\nu_1) \text{erf}(\nu_2) \),  { $\nu_1 = \frac{a_l}{ \omega(d_{HI}+d_{IU}, \hat\omega_{02})}\sqrt{\frac{\pi}{2}}$, $\nu_2 = \frac{a_l}{ \omega(d_{HI}+d_{IU}, \omega_{02})}\sqrt{\frac{\pi}{2}}$,  ${{\omega}}({d}, {\omega}_{02}) = {{\omega}}_{02} \sqrt{1 + \left( \frac{{d}\lambda}{\pi {{\omega}}_{0}^2} \right)^2 }$ is the beam waist for a Gaussian beam with initial beam waist ${{\omega}}_{0}$  and propagation distance ${d}$, $\omega_{02}$ is the initial beam waist from the HAP $\hat{\omega}_0$ is determined by solving the following equation
\begin{equation}
    \omega(d_{HI},{\hat{\omega}}_{0}) = \frac{\cos(\theta_{r})\omega(d_{HI},{\omega}_0)}{\cos(\theta_{i})}, 
\end{equation}}
\( \text{erf}(\cdot) \) is the error function. \( q_g = \left( \frac{\min\{\sigma_{u_1}^2, \sigma_{u_2}^2\}}{\max\{\sigma_{u_1}^2, \sigma_{u_2}^2\}} \right)^{1/2} \), \( \Omega = \sigma_{u_1}^2 + \sigma_{u_2}^2 \), \( \sigma_{u_1}^2 = \frac{\cos^2 \theta_r}{\cos^2 \theta_i} \sigma_s^2 + \frac{\sin^2(\theta_i + \theta_r)}{\cos^2 \theta_i} \sigma_r^2 + \sigma_l^2 \),  \( \sigma_{u_2}^2 = \sigma_s^2 + \sigma_l^2 \),  where \( \sigma_s^2 \), \( \sigma_r^2 \), and \( \sigma_l^2 \) represent the fluctuations in the positions of the laser source, OIRS and receiving lens respectively. \( I_0(\cdot) \) denotes the modified Bessel function of the first kind. { 
The expression in \eqref{PDFhg2}  complicates analytical integration and renders closed-form performance analysis intractable. To address this issue, we propose the following approximation for the PDF of $h_{g2}$.
\begin{theorem}
The PDF of $h_{g2}$ can be approximated as
\begin{equation}
\label{A2}
\begin{aligned}
    f_{h_{g2}}(h_{g2}) \approx&\,\, \frac{\varpi \mathcal{N}}{A_{02}} \left( \frac{h_{g2}}{A_{02}} \right) \left( \frac{(1+q_g^2)\varpi}{2q_g} \right)^{-1} \\&\times\sum_{k=0}^{N_k} \frac{1}{k! \Gamma(1+k)} \left( \frac{(1 - q_g^2)\varpi}{4q_g} \ln \left( \frac{h_{g2}}{A_{02}} \right) \right)^{2k},
    \end{aligned}
\end{equation} 
where $ \label{mathcalN}
    \mathcal{N} = \left\{ \sum_{k=0}^{N_k} \frac{2 q_g \Gamma(1 + 2k)}{k! \Gamma(1 + k)(1 + q_g^2)} \left[ \frac{(1 - q_g^2)\varpi}{2(1 + q_g^2)\varpi} \right]^{2k} \right\}^{-1}$.
\begin{proof}
    See Appendix \ref{A0}.
\end{proof}
\end{theorem}}

The atmospheric turbulence is also modeled using the GG distribution, as provided in \cite[Eq.(56), pp.462]{Laserbeam}
\begin{align}
\label{pdfha2}
    f_{h_{a 2}}(h_{a 2})=\frac{2(\alpha_2 \beta_2)^{\frac{\alpha_2+\beta_2}{2}}h_{a2}^{\frac{\alpha_2+\beta_2}{2}-1}}{\Gamma(\alpha_2) \Gamma(\beta_2)}  K_{\alpha_2-\beta_2}\!\left(2 \sqrt{\alpha_2 \beta_2 h_{a 2}}\right)\!,
\end{align}
where $\alpha_2=1/[\exp(\sigma_{\ln X2}^2)-1]$, and $\beta_2=1/[\exp(\sigma_{\ln Y2}^2 )-1]$. The large-scale log variance $\sigma_{\ln X2}^2$ and the small-scale log variance $\sigma_{\ln Y2}^2$ are provided by \cite[Eqs. (7) and (8)]{ata2022haps} as
\begin{align}
    \sigma_{\ln X2}^2=\frac{0.49\sigma_{B2}^2}{[1+1.11\sigma_{B2}^{12/5} ]^{7/6}},    \sigma_{\ln Y2}^2=\frac{0.51\sigma_{B2}^2}{[1+0.69\sigma_{B2}^{12/5} ]^{5/6}},
\end{align} 
where $\sigma_{B2}^2$ represents the Rytov variance. For the downlink propagation case, $\sigma_{B2}^2$ is expressed as \cite[Eq.~(11)]{ata2022haps}
\begin{align}
       \nonumber  \sigma_{B2}^{2} = \,\,&2.25 {k_w}^{7 / 6}\left(H_H-H_I\right)^{5 / 6} \sec ^{11 / 6}(\zeta_2) \\ 
       \nonumber &\times\operatorname{Re}\Big(\int_{H_I}^{H_H} C_{n}^{2}(l){\left( \frac{l-H_I}{H_H-H_I}\right)}^{5/6} \mathrm{d} l\Big) \nonumber \\&+2.25 {k_w}^{7 / 6}\left(H_I-H_U\right)^{5 / 6} \sec ^{11 / 6}(\zeta_3) \\ 
       \nonumber &\times\operatorname{Re}\Big(\int_{H_U}^{H_I} C_{n}^{2}(l){\left( \frac{l-H_U}{H_I-H_U}\right)}^{5/6} \mathrm{d} l\Big) \nonumber.
\end{align}
\begin{corollary}
The closed-form expression for the PDF of the channel gain of the FSO link from the HAP to the User, \( h_2 \), can be approximated by
\begin{equation}
\scalebox{1}{$\begin{aligned}
\label{PDFh2}
f_{h_2}&(h_2)= \frac{\varpi \mathcal{N}}{h_2 \Gamma(\alpha_2)\Gamma(\beta_2)} \sum_{k=0}^{N_k} \frac{\Gamma(1+2k)}{k! \Gamma(1+k)} \left( \frac{(1 - q_g^2) \varpi}{4 q_g} \right)^{2k}\\& \times{\rm G}_{2k+1, 2k+3}^{2k+3,0} \left[ \frac{\alpha_2 \beta_2 h_2}{A_{02} h_{p2}} \middle|
\begin{array}{c}
 {\left\{\frac{(1 + q_g^2) \varpi}{2 q_g} +1\right\}}_{2k+1}\\\alpha_2, \beta_2, {\left\{\frac{(1 + q_g^2) \varpi}{2 q_g} \right\}}_{2k+1}
\end{array} 
\right],
\end{aligned}$}
\end{equation}
where  \( \left\{a \right\}_{2k+1} \) means there are \( 2k+1 \) instances of \( a \). Strictly speaking, (\ref{PDFh2}) is an approximate expression, and the accuracy depends on the value of \( N_k \). The larger \( N_k \) is, the more accurate the expression becomes.
\end{corollary}
\begin{proof}
See Appendix \ref{A}.
\end{proof}
\subsubsection{ SNR Statistics from HAP to User}
Let \( s_2(t) \) be the signal transmitted by the HAP. The received signal by the User can be written as
\begin{align}
    y_{H2}(t)=\sqrt{\eta_2 P_{02} }h_2 s_2(t) + n_{U}(t),
\end{align}
where \( P_{02} \) is the transmit power, and \( \eta_2 \) is the optical-to-electrical conversion factor.  \( n_{U}(t) \sim \mathcal{N}\left(0, \sigma_{U}^{2}\right) \) represents Gaussian noise. Therefore, the SNR of the FSO link \( \gamma_{H2} \) can be expressed as
\begin{align}
\label{gammaHoverline2}
\gamma_{H2}=\overline{\gamma}_{H2} h_2^{r_2},
\end{align}
where \( \overline{\gamma}_{H2} = \frac{\left(\eta_2 P_{02}\right)^{r_2/2}}{\sigma_{U}^{2}} \).

{ \begin{theorem}
The PDF and CDF of $\gamma_{H2}$ are given as (\ref{PDFH2}) and (\ref{CDFH2}).
\begin{proof}
    The PDF of $\gamma_{H2}$ can be obtained from (\ref{PDFh2}) by applying the random variable transformation in (\ref{gammaHoverline2}) as (\ref{PDFH2}). Then, by using (\ref{PDFH2}) and applying \cite[Eq.~(2.24.2.3)]{prudnikov1}, we get the CDF of $\gamma_{H2}$ as (\ref{CDFH2}).
\end{proof}
\end{theorem}}

\begin{figure*}
{\begin{equation} 
\label{PDFH2}
\scalebox{0.93}{$\begin{aligned}
f_{\gamma_{H2}}(\gamma_{H2}) = \frac{\varpi \mathcal{N}}{r_2 \gamma_{H2} \Gamma(\alpha_2) \Gamma(\beta_2)} \sum_{k=0}^{N_k} \frac{\Gamma(1 + 2k)}{k! \Gamma(1 + k)} \left[ \frac{(1 - q_g^2) \varpi}{4q_g} \right]^{2k}
\rm{G}_{2k+1, 2k+3}^{2k+3, 0} \left[ \frac{\alpha_2 \beta_2}{A_{02} h_{p2}} \left( \frac{\gamma_{H2}}{\overline{\gamma}_{H2}} \right)^{\frac{1}{r_2}} \middle| 
\begin{array}{c} 
{\left\{\frac{(1 + q_g^2) \varpi}{2 q_g} +1\right\}}_{2k+1} \\ \alpha_2, \beta_2, {\left\{\frac{(1 + q_g^2) \varpi}{2 q_g} \right\}}_{2k+1}
\end{array} 
\right]
\end{aligned}$}
\end{equation}}
{\begin{equation}
\label{CDFH2}
\scalebox{0.93}{$\begin{aligned}
    &F_{\gamma_{H2}} (\gamma_{H2})= 1 - \frac{\varpi \mathcal{N}}{\Gamma(\alpha_2) \Gamma(\beta_2)} \sum_{k=0}^{N_k} \frac{\Gamma(1 + 2k)}{k! \Gamma(1 + k)} \left[ \frac{(1 - q_g^2) \varpi}{4q_g} \right]^{2k} {\rm G}_{2k+2, 2k+4}^{2k+4, 0} \left[ \frac{\alpha_2 \beta_2}{A_{02} h_{p2}} \left( \frac{\gamma_{H2}}{\overline{\gamma}_{H2}} \right)^{\frac{1}{r_2}} \middle|\!\! 
\begin{array}{c} 
1, {\left\{\frac{(1 + q_g^2) \varpi}{2 q_g} +1\right\}}_{2k+1} \\ 
0, \alpha_2, \beta_2, {\left\{\frac{(1 + q_g^2) \varpi}{2 q_g} \right\}}_{2k+1} 
\end{array} 
\!\!\!\right]
\end{aligned}$}
\end{equation}}
\hrule
\end{figure*}

\subsubsection{Performance Analysis}
A concise and unified expression for the average BER applicable to various coherent M-QAM and M-PSK modulation schemes, along with the OOK modulation technique, is given by \cite[Eq.~(21)]{dualhopFSO}
\begin{align}
\label{ABERH2}\overline{{P}}_{e2}=\delta_{B} \sum_{m=1}^{N_{B}} I_2\left(p_{B}, q_{B m}\right),
\end{align}
{ where $N_{B}$, $\delta_{B}$ , $p_{B}$, and $q_{B m}$ are detailed in Tab \ref{tab:my_label}. }
\begin{table*}[ht]
\caption{MODULATION PARAMETERS}
\centering
{ \begin{tabular}{|l|l|l|l|l|l|}
\hline \text { Modulation } & $\boldsymbol{\delta_B}$ & $\boldsymbol{p_B}$ & $\boldsymbol{q}_{\boldsymbol{Bm}}$ & $\boldsymbol{N_B}$ & \text { Detection } \\
\hline \text { M-PSK } & $\frac{2}{\max \left(\log _{2} M, 2\right)}$ & 1 / 2 & $\sin ^{2}\left(\frac{(2 k-1) \pi}{M}\right)\log _{2} M$ & $\max \left(\frac{M}{4}, 1\right)$ & \text { Heterodyne } \\
\hline \text { M-QAM } & $\frac{4}{\log _{2} M}\left(1-\frac{1}{\sqrt{M}}\right)$ & 1 / 2 &$\frac{3(2 k-1)^{2}}{2(M-1)}\log _{2} M $& $\frac{\sqrt{M}}{2}$ & \text { Heterodyne } \\
\hline \text { OOK } & 1 & 1 / 2 & 1 / 2 & 1 & \text { IM/DD } \\
\hline
\end{tabular}}
\label{tab:my_label}
\end{table*}
$I_2\left(p_{B}, q_{B k}\right)$ is defined as
\begin{equation}
\label{DI2}
\begin{aligned}
    I_2\left(p_{B}, q_{B k}\right)=&\frac{q_{Bm}^{p_{B}}}{2 \Gamma\left(p_{B}\right)} \int_{0}^{\infty} x^{p_{B}-1} \exp \left(-q_{Bm} x\right) F_{\gamma_{H2}}(x) d x.
\end{aligned}
\end{equation}

\begin{corollary}
The closed-form expression of $I_2\left(p_{B}, q_{B m}\right)$ is given as
\begin{equation}
\label{ABERIH2}
\begin{aligned}
I_2&(p_{B}, q_{B m})= \frac{1}{2}
- \frac{\varpi \mathcal{N} r_2}{2 \Gamma(p_B) \Gamma(\alpha_2) \Gamma(\beta_2)} \\&\times\sum_{k=0}^{N_k} \frac{\Gamma(1 + 2k)}{k! \Gamma(1 + k)} \left( \frac{(1 - q_g^2) \varpi}{4 q_g} \right)^{2k}
{\rm H}_{2k+3,2k+4}^{2k+4, 1} \\&\left[ \frac{\left( \frac{\alpha_2 \beta_2}{A_{02} h_{p2}} \right)^{r_2}}{\overline{\gamma}_{H2} q_{Bm}} \middle| \!\!
\text{\scriptsize$\begin{array}{c} 
(1 - p_B, 1)(1, r_2) {\left\{\left(\frac{(1 + q_g^2) \varpi}{2 q_g} +1,r_2\right)\right\}}_{2k+1}\\ 
(0, r_2) (\alpha_2, r_2)(\beta_2, r_2) {\left\{\left(\frac{(1 + q_g^2) \varpi}{2 q_g} ,r_2\right)\right\}}_{2k+1}
\end{array}$}\!\!\right].
\end{aligned}
\end{equation}
\begin{proof}
See Appendix \ref{B}.
\end{proof}
\end{corollary}

The ergodic capacity of the FSO link from HAP to User, where the FSO link operates can be expressed as presented in \cite[Eq.~(26)]{lapidoth}, as follows
\begin{equation}
\begin{aligned}\label{DefC2}
\overline{C}_2& = \int_{0}^{\infty} \ln(1+c_0\,\gamma_{H2}) f_{\gamma_{H2}}(\gamma_{H2})\,d\gamma_{H2}.
\end{aligned}
\end{equation}

\begin{corollary}
The closed-form expression of $\overline{C}_2$ is given as \eqref{C2}.

\begin{equation}
\label{C2}
\scalebox{1}{$\begin{aligned}
&\overline{C}_2 = \frac{\varpi \mathcal{N} r_2}{\Gamma(\alpha_2)\Gamma(\beta_2)} \sum_{k=0}^{N_k} \frac{\Gamma(1 + 2k) \left[ \frac{(1 - q_g^2)\varpi}{4 q_g} \right]^{2k}}{k! \Gamma(1 + k)} 
{\rm H}_{3+2k, 2k+4}^{2k+4, 1}\\& \left[ \frac{\left( \frac{\alpha_2 \beta_2}{A_{02} h_{p2}} \right)^{r_2}}{\overline{\gamma}_{H2} c_0} \middle| \!\!
\begin{array}{c} 
(0, 1)(1, r_2) {\left\{\left({(1 + q_g^2) \varpi}/{(2 q_g)} +1,r_2\right)\right\}}_{2k+1}\\ 
(0, r_2) (\alpha_2, r_2)(\beta_2, r_2) {\left\{\left(\frac{(1 + q_g^2) \varpi}{2 q_g} ,r_2\right)\right\}}_{2k+1}
\end{array}\!\!\right]
\end{aligned}$}
\end{equation}

\begin{proof}
See Appendix \ref{C}.
\end{proof}
\end{corollary}

\begin{corollary}
The $s$-th moments of $\gamma_{H2}$ can be demonstrated as \eqref{MonmentsH2}.
{\begin{equation}
\label{MonmentsH2}
\scalebox{1}{$\begin{aligned}
\mathbb{E}(\gamma_{H2}^s) =\,& \frac{\varpi \mathcal{N}}{ \Gamma(\alpha_2) \Gamma(\beta_2)} \sum_{k=0}^{N_k} \frac{\Gamma(1 + 2k) \left[ \frac{(1 - q_g^2)\varpi}{4 q_g} \right]^{2k}}{k! \Gamma(1 + k)\left[{ sr_2 + \frac{(1 + q_g^2) \varpi}{2 q_g} }\right]^{1 + 2k}}\\&\times \Gamma(sr_2 + \alpha_2) \Gamma(sr_2 + \beta_2)
 \left[ \frac{A_{02} h_{p2}{\overline{\gamma}_{H2}^{\frac{1}{r_2}}}}{\alpha_2 \beta_2}   \right]^{sr_2}
\end{aligned}$}
\end{equation}}
\begin{proof}
The $s$-th moments of $\gamma_{H2}$ is given as
\begin{equation}
\label{D1}
    E(\gamma_{H2}^s) = \int_{0}^{\infty} \gamma_{H2}^s f_{\gamma_{H2}}(\gamma_{H2}) \, \mathrm{d}\gamma_{H2}.
\end{equation}
Substituting (\ref{PDFH2}) into  (\ref{D1}) and applying (1.5) of \cite{Mittal}, (\ref{MonmentsH2}) can be obtained.
\end{proof}
\end{corollary}
\section{END-TO-END SYSTEM PERFORMANCE}
\subsection{End-to-End SNR Statistics}
The end-to-end SNR for the fixed-gain relaying scheme can be derived using the expression provided in \cite[Eq.~(28)]{fixedgain}, under the assumption that the effects of saturation can be neglected, as follows
\begin{align}\label{E2ESNR}
\gamma=\frac{\gamma_{H1} \gamma_{H2}}{\gamma_{H2}+C},
\end{align}
where $C$ represents a constant relay gain.

\begin{theorem}
The CDF and PDF of the overall SNR can be derived in terms of the bivariate Fox-H function, also referred to  as \eqref{PDFgamma} and \eqref{CDFgamma}, also known as the Fox's H-function of two variables \cite{Mittal}, whose MATLAB implementation is provided in \cite{Peppas}.
\begin{proof}
See Appendix \ref{E}.
\end{proof}
\end{theorem}
\subsection{Performance Analysis}
\subsubsection{Outage Probability}
The OP represents the probability that the end-to-end SNR falls below a specified threshold \(\gamma_{\rm{th}}\). By substituting \(\gamma\) with \(\gamma_{\rm{th}}\) in (\ref{CDFgamma}), a unified expression for the OP in both detection methods can be directly obtained.
\begin{figure*}
\begin{equation}
\label{PDFgamma}
\scalebox{1}{$
\begin{aligned}
f_{\gamma}(\gamma) =\,& \frac{\eta_s^2 \varpi \mathcal{N}}{   \Gamma(\alpha_1) \Gamma(\beta_1)   \Gamma(\alpha_2) \Gamma(\beta_2)\gamma} \sum_{k=0}^{N_k} \frac{\Gamma(1 + 2k)}{k! \Gamma(1 + k)}\left[ {(1 - q_g^2)\varpi}/{(4 q_g)} \right]^{2k}
\\& \times{\rm H}^{0,1;3,0;2k+4,0}_{1,0;2,3;2k+1,2k+4}\left[\!\!\!\text{\small$\begin{array}{c}\left( 1; -1, 1 \right) \\-\\
\left( 1 + \eta_s^2  , r_1 \right) \left( 0, 1 \right) \\
\left( \eta_s^2  , r_1 \right) (\alpha_1, r_1) (\beta_1, r_1) \\{\left\{\left({(1 + q_g^2) \varpi}/{(2 q_g)}+1,r_2\right)\right\}}_{2k+1}\\(\alpha_2, r_2) (\beta_2, r_2)(0,1){\left\{\left({(1 + q_g^2) \varpi}/{(2 q_g)} ,r_2\right)\right\}}_{2k+1}\end{array}$}\!\!\!\middle|\!\!\!\text{\small$\begin{array}{c} \left(\frac{\alpha_1 \beta_1}{A_{01} h_{p1}}\right) ^{r_1}\frac{\gamma}{\overline{\gamma}_{H1}}\\\\
 \left(\frac{\alpha_2 \beta_2}{A_{02} h_{p2}}\right) ^{r_2}\frac{C}{\overline{\gamma}_{H2}} \end{array}$}
\!\!\!\right]\!\!,
\end{aligned}$}
\end{equation}
\begin{equation}
\label{CDFgamma}\scalebox{1}{$
\begin{aligned}
F_{\gamma}(\gamma) =\,& 1 - \frac{\eta_s^2 \varpi \mathcal{N}}{   \Gamma(\alpha_1) \Gamma(\beta_1)   \Gamma(\alpha_2) \Gamma(\beta_2)} \sum_{k=0}^{N_k} \frac{\Gamma(1 + 2k)}{k! \Gamma(1 + k)}\left[ {(1 - q_g^2)\varpi}/{(4 q_g)} \right]^{2k}\\& \times
{\rm H}^{0,1;3,0;2k+4,0}_{1,0;2,3;2k+1,2k+4}\left[\!\!\!\text{\small$\begin{array}{c}\left( 1; -1, 1 \right) \\-\\
\left( 1 + \eta_s^2  , r_1 \right) \left( 1, 1 \right) \\
\left( \eta_s^2  , r_1 \right) (\alpha_1, r_1) (\beta_1, r_1) \\{\left\{\left({(1 + q_g^2) \varpi}/{(2 q_g)}+1,r_2\right)\right\}}_{2k+1}\\(\alpha_2, r_2) (\beta_2, r_2)(0,1){\left\{\left(\frac{(1 + q_g^2) \varpi}{2 q_g} ,r_2\right)\right\}}_{2k+1}\end{array}$}\!\!\!\middle|\!\!\!\text{\small$\begin{array}{c}\left(\frac{\alpha_1 \beta_1}{A_{01} h_{p1}}\right) ^{r_1}\frac{\gamma}{\overline{\gamma}_{H1}}\\\\
 \left(\frac{\alpha_2 \beta_2}{A_{02} h_{p2}}\right) ^{r_2}\frac{C}{\overline{\gamma}_{H2}}\end{array}$}
\!\!\!\right]\!\!,
\end{aligned}$}
\end{equation}
\hrule
\end{figure*}
In (\ref{CDFgamma}), the CDF is represented using the bivariate Fox-H function, which is complex and not readily available in common mathematical software such as MATLAB or MATHEMATICA. We analyze the CDF in the high SNR regime using an asymptotic approach to overcome this limitation. This results in a simplified version of the CDF in (\ref{CDFgamma}), which only involves elementary functions that are already supported in MATLAB and MATHEMATICA, as
shown in (\ref{CDFgammaA}).

\begin{figure*}
\begin{equation}
\scalebox{1}{$\begin{aligned}
\label{CDFgammaA}
&F_{\gamma}(\gamma)\underset{\scalebox{0.8}{$\overline{\gamma}_{H1}\gg 1, \overline{\gamma}_{H2}\gg 1
$}}{\mathop{\approx}}\frac{\eta_s^2 \varpi \mathcal{N}}{   \Gamma(\alpha_1) \Gamma(\beta_1)   \Gamma(\alpha_2) \Gamma(\beta_2)} \sum_{k=0}^{N_k} \frac{\Gamma(1 + 2k)}{k! \Gamma(1 + k)} \left[ \frac{(1 - q_g^2)\varpi}{4 q_g} \right]^{2k}  
\left\{\frac{  \Gamma(\alpha_1 - \eta_s^2  ) \Gamma(\beta_1 - \eta_s^2  ) \Gamma(\alpha_2) \Gamma(\beta_2)}
{\eta_s^2   \left( \frac{(1 + q_g^2) \varpi}{2 q_g} \right)^{2k+1}}\right.\\&\times\left[ \left( \frac{\alpha_1 \beta_1}{A_{01} h_{p_1}} \right)^{r_1}  \frac{\gamma}{\overline{\gamma}_{H_1}}  \right]^{\frac{\eta_s^2  }{r_1}} + \frac{  \Gamma(\alpha_1 \beta_1) \Gamma(\alpha_2) \Gamma(\beta_2)}
{\beta_1 (\eta_s^2   - \beta_1) \left( \frac{(1 + q_g^2) \varpi}{2 q_g} \right)^{2k+1}} 
\left[\left( \frac{\alpha_1 \beta_1}{A_{01} h_{p_1}} \right)^{r_1}  \frac{\gamma}{\overline{\gamma}_{H_1}}  \right]^{\frac{\beta_1}{r_1}}\!\!\!
 \\&+ \frac{  \Gamma(\eta_s^2   - \beta_1) \Gamma(\alpha_1 - \beta_1) 
\Gamma\left(\alpha_2 - \frac{r_2 \beta_1}{r_1}\right) 
\Gamma\left(\beta_2 - \frac{r_2 \beta_1}{r_1}\right)}
{\beta_1 (\eta_s^2   - \beta_1) \left[ \left( \frac{1 + q_g^2}{2 q_g} \right) \varpi - \frac{r_2 \beta_1}{r_1} \right]^{2k+1}} \left[\left( \frac{\alpha_1 \beta_1}{A_{01} h_{p_1}} \right)^{r_1}\left( \frac{\alpha_2 \beta_2}{A_{02} h_{p_2}} \right)^{r_2} \frac{ C}{\overline{\gamma}_{H_1} \,\overline{\gamma}_{H_2}} \right]^{\frac{\beta_1}{r_1}}+  \Gamma\left(\alpha_1 - \frac{r_1}{r_2} \frac{(1 + q_g^2) \varpi}{2 q_g}\right)  \\&\times(2k + 1)\frac{ 
\Gamma\left(\beta_1 - \frac{r_1}{r_2} \frac{(1 + q_g^2) \varpi}{2 q_g}\right) 
\Gamma\left(\alpha_2 - \frac{(1 + q_g^2) \varpi}{2 q_g}\right) 
\Gamma\left(\beta_2 - \frac{(1 + q_g^2) \varpi}{2 q_g}\right)}
{\left(\eta_s^2   - \frac{r_1}{r_2} \frac{(1 + q_g^2) \varpi}{2 q_g}\right) \left(\frac{(1 + q_g^2) \varpi}{2 q_g}\right)}  \left[\left( \frac{\alpha_1 \beta_1}{A_{01} h_{p_1}} \right)^{r_1}\left( \frac{\alpha_2 \beta_2}{A_{02} h_{p_2}} \right)^{r_2} \frac{ C}{\overline{\gamma}_{H_1} \,\overline{\gamma}_{H_2}}\right]^{\frac{(1 + q_g^2) \varpi}{2 q_g r_2}}\\&+ \frac{  \Gamma\left(\eta_s^2   - \frac{r_1}{r_2} \alpha_2\right) 
\Gamma\left(\alpha_1 - \frac{r_1}{r_2} \alpha_2\right) 
\Gamma\left(\beta_1 - \frac{r_1}{r_2} \alpha_2\right) 
\Gamma(\beta_2 - \alpha_2)}
{\alpha_2 \left(\eta_s^2   - \alpha_2\right) \left[\left(\frac{1 + q_g^2}{2 q_g}\right) \varpi - \alpha_2\right]^{2k+1}}  \left[ \left( \frac{\alpha_1 \beta_1}{A_{01} h_{p_1}} \right)^{r_1}\left( \frac{\alpha_2 \beta_2}{A_{02} h_{p_2}} \right)^{r_2} \frac{C}{\overline{\gamma}_{H_1} \,\overline{\gamma}_{H_2}}  \right]^{\frac{\alpha_2}{r_2}}\\&+\left. \frac{  \Gamma\left(\eta_s^2   - \frac{r_1}{r_2} \beta_2\right) 
\Gamma\left(\alpha_1 - \frac{r_1}{r_2} \beta_2\right) 
\Gamma\left(\beta_1 - \frac{r_1}{r_2} \beta_2\right) 
\Gamma(\alpha_2 - \beta_2)}
{\beta_2 \left(\eta_s^2   - \beta_2\right) \left[\left(\frac{1 + q_g^2}{2 q_g}\right) \varpi - \beta_2\right]^{2k+1}}  \left[\left( \frac{\alpha_1 \beta_1}{A_{01} h_{p_1}} \right)^{r_1}\left( \frac{\alpha_2 \beta_2}{A_{02} h_{p_2}} \right)^{r_2} \frac{ C}{\overline{\gamma}_{H_1} \,\overline{\gamma}_{H_2}}  \right]^{\frac{\beta_2}{r_2}}\right\} 
\end{aligned}$}
\end{equation}
\hrule
\end{figure*}

\begin{proof}
See Appendix \ref{F}.
\end{proof}

This asymptotic expression proves particularly useful for determining the system's diversity order. Specifically, the diversity gain of our system can be calculated as
\begin{align}
\label{diversity}
\mathcal{G}_{d} = \min \left(\frac{\alpha_1}{r_1}, \frac{\beta_1}{r_1}, \frac{\eta_{S}^{2}}{r_1},\frac{\alpha_2}{r_2} ,\frac{\beta_2}{r_2}, \frac{(1 + q_g^2) \varpi}{2 q_g r_2} \right).
\end{align}

\subsubsection{Average Bit-Error Rate}
\begin{corollary}
The expression for the average BER applicable to various coherent M-QAM and M-PSK modulation schemes is given by \cite[Eq.~(21)]{dualhopFSO}
\begin{align}\label{ABERgamma}
\overline{{P}}_{e}=\delta_{B} \sum_{m=1}^{N_{B}} I\left(p_{B}, q_{B m}\right),
\end{align}
where the expression of $I\left(p_{B}, q_{B m}\right)$ is given in \eqref{Igamma}. To overcome the challenges posed by the bivariate Fox-H function, an asymptotic result of (\ref{Igamma}) is given as  (\ref{IgammaA}).
\begin{figure*}
\vspace{-0.5cm}
\begin{equation}
\label{Igamma}
\scalebox{1}{$\begin{aligned}
I\left(p_{B}, q_{B m}\right) =\,& \frac{1}{2} - \frac{\eta_s^2 \varpi \mathcal{N}}{2 \Gamma(p_B)\Gamma(\alpha_1) \Gamma(\beta_1)   \Gamma(\alpha_2) \Gamma(\beta_2)} \sum_{k=0}^{N_k} \frac{\Gamma(1 + 2k)}{k! \Gamma(1 + k)} \left[ {(1 - q_g^2)\varpi}/{(4 q_g)} \right]^{2k}
\\&\times{\rm H}^{0,1;3,1;2k+4,0}_{1,0;3,3;2k+1,2k+4}\left[\!\!\!\!\text{\small$\begin{array}{c}\left( 1; -1, 1 \right) \\-\\(1-p_B,1)
\left( 1 + \eta_s^2  , r_1 \right) \left( 1, 1 \right) \\
\left( \eta_s^2  , r_1 \right) (\alpha_1, r_1) (\beta_1, r_1) \\{\left\{\left({(1 + q_g^2) \varpi}/{(2 q_g)}+1,r_2\right)\right\}}_{2k+1}\\(\alpha_2, r_2) (\beta_2, r_2)(0,1){\left\{\left(\frac{(1 + q_g^2) \varpi}{2 q_g} ,r_2\right)\right\}}_{2k+1}\end{array}$}\!\!\middle|\!\!\!\text{\small$\begin{array}{c}\left(\frac{\alpha_1 \beta_1}{A_{01} h_{p1}}\right) ^{r_1}\frac{1}{\overline{\gamma}_{H1}q_{Bm}}\\\\
\left(\frac{\alpha_2 \beta_2}{A_{02} h_{p2}}\right) ^{r_2}\frac{ C}{\overline{\gamma}_{H2}}\end{array}$}
\!\!\!\right]\!  ,
\end{aligned}$}
\end{equation}
\hrule
{\begin{equation}
\label{IgammaA}
\scalebox{1}{$\begin{aligned}
&I(p_B, q_{Bm})\underset{\scalebox{0.8}{$\begin{array}{c}
\overline{\gamma}_{H1}\gg 1,\\\overline{\gamma}_{H2}\gg 1
\end{array}
$}}{\mathop{\approx}}\frac{\eta_s^2 \varpi \mathcal{N}}{2 \Gamma(\alpha_1) \Gamma(\beta_1)   \Gamma(\alpha_2) \Gamma(\beta_2) \Gamma(p_B)}\sum_{k=0}^{N_k} \frac{\Gamma(1 + 2k)}{k! \Gamma(1 + k)} \left[ \frac{(1 - q_g^2)\varpi}{4 q_g} \right]^{2k}  \!\!\left\{\frac{\Gamma(\alpha_1 - \eta_s^2  ) \Gamma(\beta_1 - \eta_s^2  ) \Gamma(\alpha_2) \Gamma(\beta_2) }
{\eta_s^2   \left[\left(\frac{1 + q_g^2}{2 q_g}\right) \varpi\right]^{2k+1} } \right. \\& \times \Gamma\left(p_B + \frac{\eta_s^2  }{r_1}\right)\left[\left( \frac{\alpha_1 \beta_1}{A_{01} h_{p_1}} \right)^{r_1}\frac{1}{{q_{Bm}\overline{\gamma}_{H_1}}}\right]^{\eta_s^2} + \frac{ \Gamma(\beta_1 - \alpha_1) \Gamma(\alpha_2) \Gamma(\beta_2) \Gamma\left(p_B + \frac{\alpha_1}{r_1}\right)}
{\alpha_1 (\eta_s^2   - \alpha_1) \left[\left(\frac{1 + q_g^2}{2 q_g}\right) \varpi\right]^{2k+1} } \left[ \!{\left( \frac{\alpha_1 \beta_1}{A_{01} h_{p_1}} \right)^{r_1}} \frac{1}{{q_{Bm}\overline{\gamma}_{H_1}}}\right]^{\frac{\alpha_1}{r_1}} \\&+ 
\frac{ \Gamma(\alpha_1 -\beta_1) \Gamma(\alpha_2) \Gamma(\beta_2) \Gamma\left(p_B + \frac{\beta_1}{r_1}\right)}
{\beta_1 (\eta_s^2   - \beta_1) \left[\left(\frac{1 + q_g^2}{2 q_g}\right) \varpi\right]^{2k+1} } \left[ \left( \frac{\alpha_1 \beta_1}{A_{01} h_{p_1}} \right)^{r_1}\frac{1}{{q_{Bm}\overline{\gamma}_{H_1}}}\right]^{\frac{\beta_1}{r_1}} +\left[\left( \frac{\alpha_1 \beta_1}{A_{01} h_{p_1}} \right)^{r_1}\left( \frac{\alpha_2 \beta_2}{A_{02} h_{p_2}} \right)^{r_2}\frac{ C}{\overline{\gamma}_{H_1} \overline{\gamma}_{H_2}q_{Bm}}  \right]^{\frac{(1 + q_g^2) \varpi}{2 q_g r_2}} \\&\times (2k + 1)\frac{  \Gamma\left(\alpha_1 - \frac{r_1}{r_2} \frac{(1 + q_g^2) \varpi}{2 q_g}\right) 
\Gamma\left(\beta_1 - \frac{r_1}{r_2} \frac{(1 + q_g^2) \varpi}{2 q_g}\right) 
\Gamma\!\left(\alpha_2 - \frac{(1 + q_g^2) \varpi}{2 q_g}\right) 
\Gamma\left(\beta_2 - \frac{(1 + q_g^2) \varpi}{2 q_g}\right) \Gamma\left(p_B + \frac{(1 + q_g^2) \varpi}{2 q_g r_2}\right)
}
{\left(\eta_s^2   - \frac{r_1}{r_2} \frac{(1 + q_g^2) \varpi}{2 q_g}\right) 
\left(\frac{(1 + q_g^2) \varpi}{2 q_g}\right) }\\&  + \frac{ \Gamma\left(\eta_s^2   - \frac{r_1}{r_2} \alpha_2\right) 
\Gamma\left(\alpha_1 - \frac{r_1}{r_2} \alpha_2\right) 
\Gamma\left(\beta_1 - \frac{r_1}{r_2} \alpha_2\right) 
\Gamma(\beta_2 - \alpha_2) \Gamma\left(p_B + \frac{\alpha_2}{r_2}\right)}
{\alpha_2 \left(\eta_s^2   - \alpha_2\right) 
\left[\left(\frac{1 + q_g^2}{2 q_g}\right) \varpi - \alpha_2\right]^{2k+1} }  \left[ \left( \frac{\alpha_1 \beta_1}{A_{01} h_{p_1}} \right)^{r_1}\left( \frac{\alpha_2 \beta_2}{A_{02} h_{p_2}} \right)^{r_2} \frac{C}{\overline{\gamma}_{H_1} \,\overline{\gamma}_{H_2}q_{Bm}}\right]^{\frac{\alpha_2}{r_2}}\\&+\left. \frac{ \Gamma\left(\eta_s^2   - \frac{r_1}{r_2} \beta_2\right) 
\Gamma\left(\alpha_1 - \frac{r_1}{r_2} \beta_2\right) 
\Gamma\left(\beta_1 - \frac{r_1}{r_2} \beta_2\right) 
\Gamma(\alpha_2 - \beta_2) 
\Gamma\left(p_B + \frac{\beta_2}{r_2}\right)}
{\beta_2 (\eta_s^2   - \beta_2) \left[\left(\frac{1 + q_g^2}{2 q_g}\right) \varpi - \beta_2\right]^{2k+1} } \left[ {\left( \frac{\alpha_1 \beta_1}{A_{01} h_{p_1}} \right)^{r_1}\left( \frac{\alpha_2 \beta_2}{A_{02} h_{p_2}} \right)^{r_2} \frac{C}{q_{Bm}\overline{\gamma}_{H_1} \,\overline{\gamma}_{H_2}}} \right]^{\frac{\beta_2}{r_2}}\right\}  
\end{aligned}$}
\end{equation}}
\end{figure*}
\begin{proof}
See Appendix \ref{G}.
\end{proof}
\end{corollary}

\begin{figure*}
\begin{equation}
\label{Capacitygamma}
\begin{aligned}
&\overline{C} =  \frac{\eta_s^2 \varpi \mathcal{N}}{   \Gamma(\alpha_1) \Gamma(\beta_1)   \Gamma(\alpha_2) \Gamma(\beta_2)} \sum_{k=0}^{N_k} \frac{\Gamma(1 + 2k)}{k! \Gamma(1 + k)}\left[ {(1 - q_g^2) \varpi}/{(4 q_g)} \right]^{2k}\\& \times{\rm H}^{0,1;4,1;2k+4,0}_{1,0;3,4;2k+1,2k+4}\left[\!\!\!\!\begin{array}{c}\left( 1; -1, 1 \right) \\-\\
(0,1)\left( 1 + \eta_s^2  , r_1 \right) \left( 1, 1 \right) \\
\left( \eta_s^2  , r_1 \right) (\alpha_1, r_1) (\beta_1, r_1)(0,1) \\{\left\{\left({(1 + q_g^2) \varpi}/{(2 q_g)}+1,r_2\right)\right\}}_{2k+1}\\(\alpha_2, r_2) (\beta_2, r_2)(0,1){\left\{\!\!\left(\frac{(1 + q_g^2) \varpi}{2 q_g} ,r_2\right)\!\!\right\}}_{2k+1}\end{array}\!\!\!\middle|\!\!\!\begin{array}{c}\left(\frac{\alpha_1 \beta_1}{A_{01} h_{p1}}\right) ^{r_1}\frac{1}{\overline{\gamma}_{H1}c_0}\\\\
 \left(\frac{\alpha_2 \beta_2}{A_{02} h_{p2}}\right) ^{r_2}\frac{C}{\overline{\gamma}_{H2}}\end{array}
\!\!\!\!\right]\!\!   
\end{aligned}
\end{equation}
\hrule
\end{figure*}
\subsubsection{Ergodic Capacity}
\begin{corollary}
The ergodic capacity of the end-to-end system can be expressed as  \eqref{Capacitygamma}.

\end{corollary}

\begin{proof}
See Appendix \ref{H}.
\end{proof}

\subsubsection{$s$-th moments}

{ The {$s$-th moment} of  $\gamma$ is defined as
\begin{equation}
\mathbb{E}[\gamma^s] = \int_0^\infty \gamma^s f_{\gamma}(\gamma)\, d\gamma,
\end{equation}}
\begin{corollary}
The $s$-th moments of $\gamma$ can be demonstrated as
\begin{figure*}
\begin{equation}
\label{moment}
\scalebox{1}{$\begin{aligned}
E(\gamma^s) = \,&\frac{\eta_s^2 \varpi \mathcal{N}}{\Gamma(\alpha_1) \Gamma(\beta_1) \Gamma(\alpha_2) \Gamma(\beta_2)} \sum_{k=0}^{N_k} \frac{\Gamma(1 + 2k)}{k! \, \Gamma(1 + k)} \left( \frac{(1 - q_g^2) \varpi}{4 q_g} \right)^{2k} \frac{\Gamma( \eta_s^2   + r_1s) \Gamma( \alpha_1 + r_1s) \Gamma( \beta_1 + r_1s)}{\Gamma(1 +  \eta_s^2   + r_1s) \Gamma(s)} \\&\times\left[ \left( \frac{\alpha_1 \beta_1}{A_{01} h_{p1}} \right)^{r_1} \frac{1}{ \overline{\gamma}_{H1} }\right]^{-s}{\rm H}_{2k+2,2k+4}^{2k+4,1}  \left[ \!\!\!\!
\begin{array}{c}
(1 - s, 1) \left( 1 + \frac{(1 + q_g^2) \varpi}{2 q_g}, r_2 \right)_{2k+1} \\
(\alpha_2, r_2) (\beta_2, r_2) (0, 1) \left( \frac{(1 + q_g^2) \varpi}{2 q_g}, r_2 \right)_{2k+1} 
\end{array}
\!\!\!\middle|\!  \left( \frac{\alpha_2 \beta_2}{A_{02} h_{p2}} \right)^{r_2}\frac{C}{\overline{\gamma}_{H2} }\right].
\end{aligned}$}
\end{equation}
\hrule
\end{figure*}
\end{corollary}
\begin{proof}
See Appendix \ref{I}.
\end{proof}

\section{NUMERICAL ANALYSIS}

{ Table~\ref{tab} outlines the simulation parameters and default settings adopted throughout the numerical evaluation. Unless otherwise specified, all results are generated using the values listed in Table~\ref{tab}.} Analytical results are presented alongside Monte Carlo (MC) simulations, and the comparison shows a strong alignment between the derived analytical expressions and the simulation results, confirming the accuracy of our findings.
\begin{table}[!ht]\centering
\setlength{\tabcolsep}{4pt}
\caption{System Parameters}
\begin{tabular}{cccc}
\hline
Parameters & Values                                          & Parameters    & Values                                           \\ \hline
$H_O$      & 10 m                                            & $H_H$         & 18 km                                            \\
$H_I$      & 80 m                                            & $H_U$         & 1 m                                              \\
$Y_{H}$    & 0                                               & $Y_{I}$       & -1000 m                                          \\
$\theta_i$  & ${\pi}/{6}$ & $\theta_r$     & ${3\pi}/{8}$ \\
$Y_{U}$    & -1020 m                                         &   $\omega(d_{HI}, \omega_0)$            & $4a_l$                                                 \\
$r_a$      & 5 mm                                            & $\omega$      & 30 m/s                                           \\
$\omega_b$ & 3$r_a$                                          & $F_0$         & $\infty$                                      \\
V          & 10 km                                           & $\sigma_{S0}$    & $r_a$                                            \\
$\lambda$  & 1550 nm                                         & $\omega_{01}$         & 1 mm                                             \\
$A$        & $1.7\times 10^{-13} \text{m}^{-2/3}$            & $\gamma_{th}$ & 2 dB                                             \\
$\zeta_p$  & 1                                               & $\kappa$      & $0.43\times 10^{-3}$                             \\
$a_l$      & 2.5 mm                                          &    $ \phi_r  $          &  $\pi$   \\
\( \theta_{rl}  \)&0& $(\sigma_s, \sigma_r, \sigma_l)$ & $0.5 (a_l, a_l, a_l)$\\ $\zeta_1$ & $60^\circ$ &$N_k$ & 5\\   \hline                  
\end{tabular}\label{tab}
\end{table}
\subsection{HAP to User link}
For the HAP-to-User channel via OIRS, an approximation was applied to the 3D GML model with a Hoyt distribution as proposed in \cite{OIRS1}. Therefore, it is crucial to determine the minimum value of $N_x$ that ensures the approximation achieves the desired level of accuracy. Fig.\ref{hg2A} presents a comparison between the exact and approximate PDF of GML under different fluctuation conditions.

From Fig. \ref{hg2A}, we can observe that when the fluctuation levels of the laser source, OIRS, and receiving lens are all minimal (namely $(\sigma_s, \sigma_r, \sigma_l) = 0.5 (a_l, a_l, a_l)$), setting \( N_x = 0 \) already provides an excellent approximation. To determine which of these fluctuations has the greatest impact on the approximation error, we increased the fluctuation levels for each component individually by the same amount. For both the laser source and the receiver lens, even with higher fluctuation levels (namely $(\sigma_s, \sigma_r, \sigma_l) = 0.5 (2a_l, a_l, a_l)$ and $(\sigma_s, \sigma_r, \sigma_l) = 0.5 (a_l, a_l, 3a_l)$), \( N_x = 0 \) continues to yield a highly accurate approximation. However, when the fluctuations of  OIRS increased ($(\sigma_s, \sigma_r, \sigma_l) = 0.5 (a_l, 2a_l, a_l)$), the approximation with \( N_x = 0 \) shows a noticeable deviation from the exact result. Increasing \( N_x \) to 5, however, restores a high level of accuracy in the approximation. This indicates that fluctuations in the OIRS have the most significant impact on the approximation error.
\begin{figure}[!ht]
\centering\includegraphics[width=0.5\textwidth, trim={20 25 25 25}, clip]{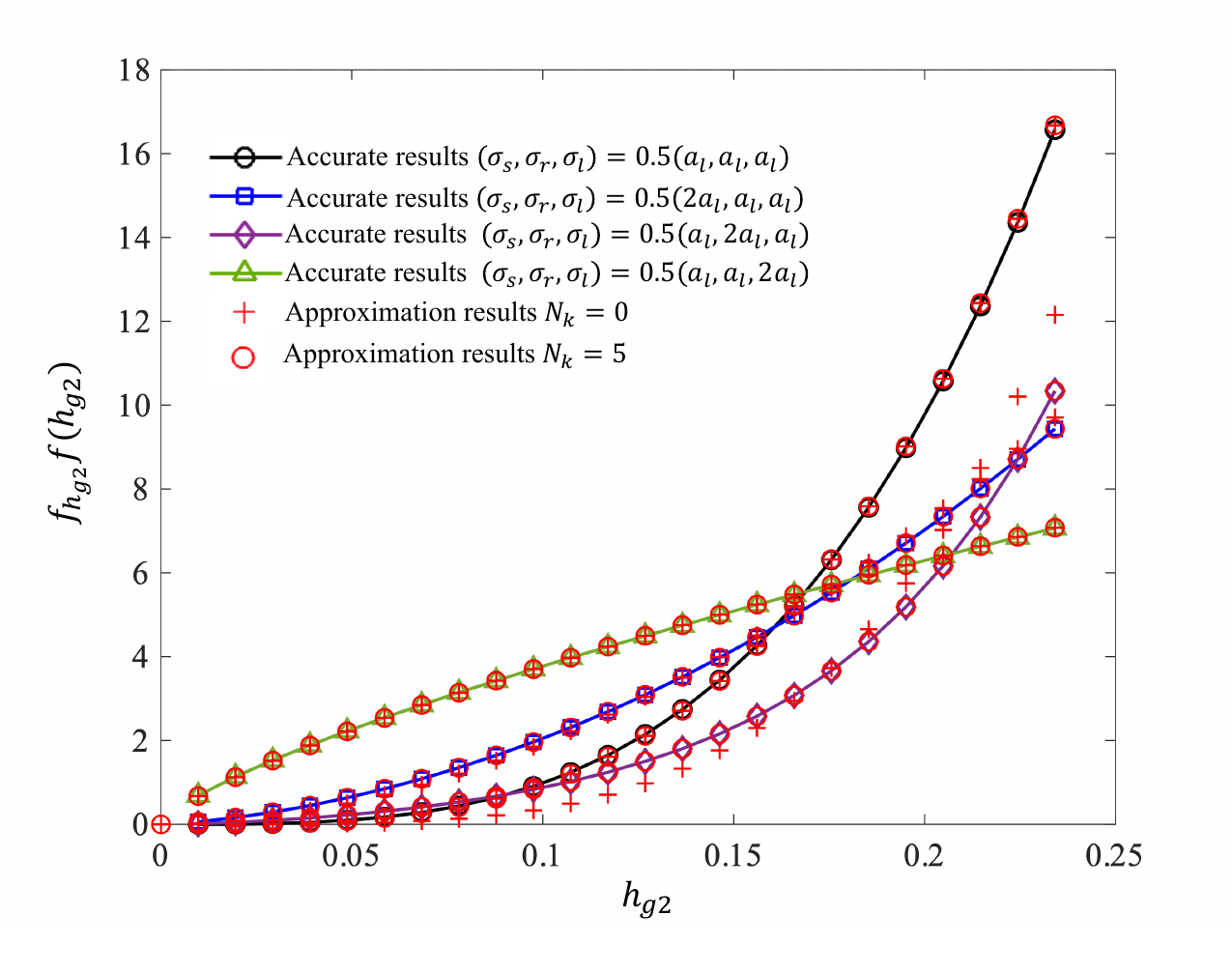}
    \caption{Comparison plot of the exact and approximate PDF of GML under varying fluctuation conditions.}
    \label{hg2A}
\end{figure}

{
Fig.~\ref{hg2A2} illustrates the convergence performance of the truncated series approximation for the PDF of the GML, evaluated under various incident angles $\theta_i$. The vertical axis represents the average $\ell_2$-norm error, denoted by \( \mathcal{E}_{\ell_2} \), between the exact expression in \eqref{PDFhg2} and the approximate expression in \eqref{A2}, while the horizontal axis denotes the number of truncation terms $N_k$. This error metric is computed as:
\begin{equation}
\mathcal{E}_{\ell_2} = \frac{1}{M} \sum_{i=1}^{M} \left( f_{\text{exact}}(x_i) - f_{\text{approx}}(x_i) \right)^2,
\end{equation}
where \( M = 50 \) uniformly spaced points \( \{x_i\}_{i=1}^{M} \) are selected over the support of the distribution. As expected, increasing $N_k$ leads to a rapid reduction in \( \mathcal{E}_{\ell_2} \), demonstrating the accuracy and efficiency of the proposed approximation. Moreover, the convergence behavior depends on the incident angle $\theta_i$; smaller angles (e.g., $\pi/12$) exhibit faster convergence, while larger angles (e.g., $\pi/3$) require more terms to achieve similar accuracy.
\begin{figure}[!ht]
\centering\includegraphics[width=0.47\textwidth]{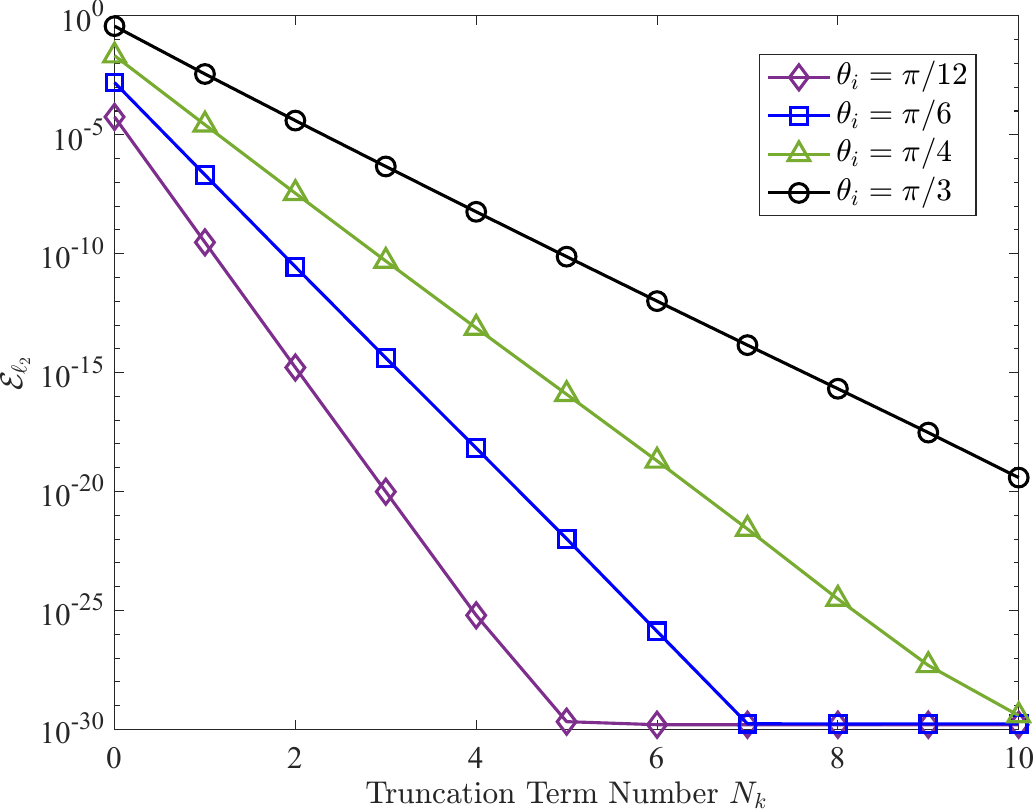}
\caption{{ Convergence behavior of the truncated series approximation for the PDF of GML under different incident angles $\theta_i$.}}
\label{hg2A2}
\end{figure}}

Then we proceed with analyzing various system performance metrics.  Fig. \ref{OP1} illustrates the variation in the OP of the HAP-to-User channel via OIRS concerning the average SNR of the second hop $\overline{\gamma}_{H2}$, under different fluctuation conditions for both two detection methods. 
\begin{figure}[!ht]
\centering\includegraphics[width=0.5\textwidth, trim={20 20 25 25}, clip]{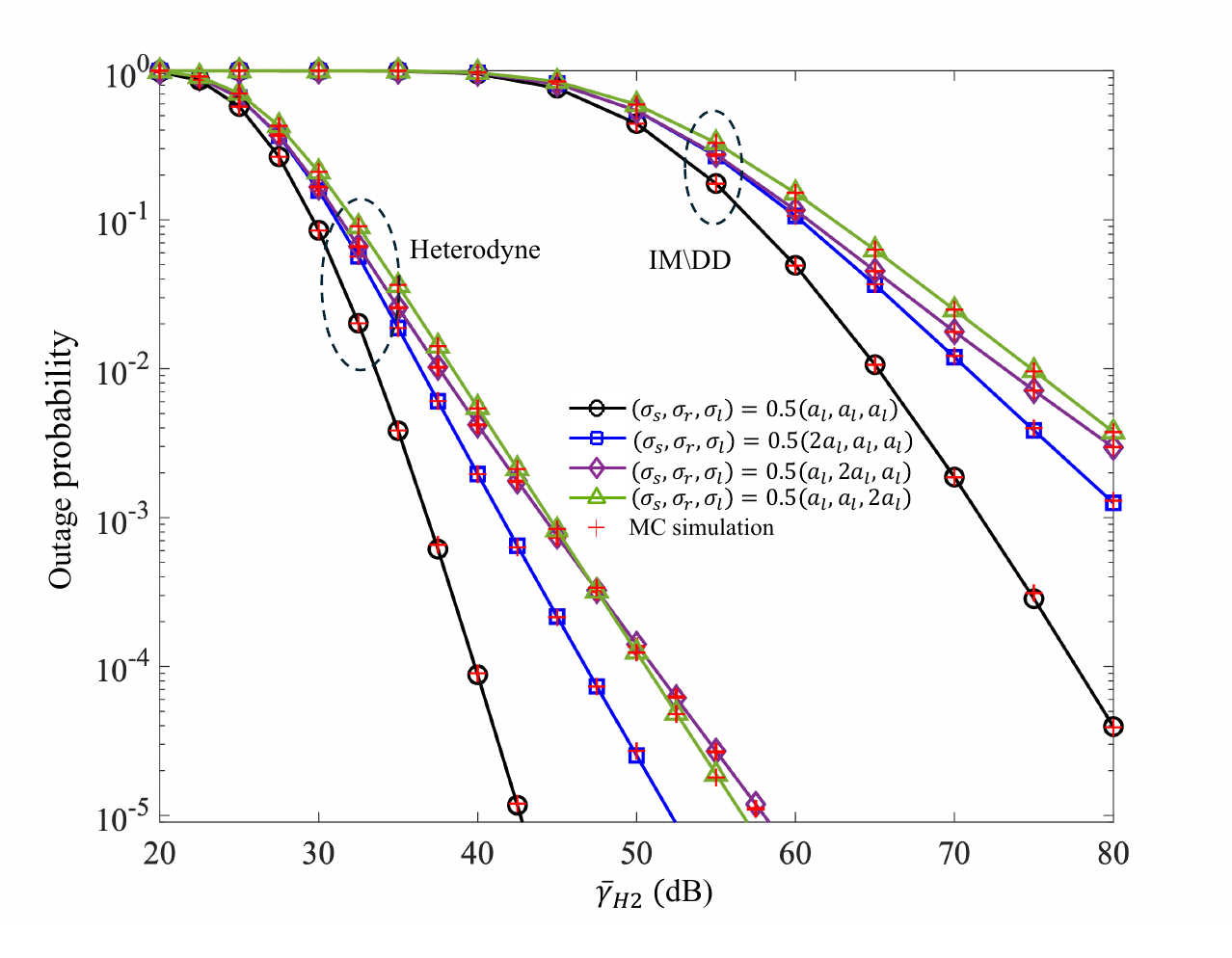}
    \caption{The OP of the HAP-to-User channel via OIRS is analyzed as a function of the SNR under different fluctuation conditions for both detection methods.}
    \label{OP1}
\end{figure}
When the fluctuation levels of the laser source, OIRS, and receiving lens increase uniformly, the system's OP also increases.  The results reveal that fluctuations in the laser source have the least impact on the system's OP, while fluctuations in the OIRS and receiving lens have a nearly identical effect on the system's OP. Specifically, when the SNR \(\gamma_{H2}\) is 40 dB, the OP corresponding to four different fluctuation levels is \(9.0 \times 10^{-5}\), \(2.0 \times 10^{-3}\), \(4.2 \times 10^{-3}\), and \(5.4 \times 10^{-3}\), respectively.

After analyzing the OP, we examine the differences in average BER for various modulation schemes, as illustrated in Fig. \ref{ABER1}, which depicts its variation with system SNR. Overall, OOK has the highest average BER, primarily because it employs the IM/DD detection technique, which shows that using the heterodyne detection technique can significantly enhance the system's performance. For M-QAM and M-PSK, the average BER increases as \( M \) increases. This is because increasing  \( M \)enhances the spectral efficiency of data transmission. With the same bandwidth, a larger \( M \) results in a higher data transmission rate. However, higher transmission rates also lead to a greater probability of errors.
\begin{figure}[!ht]
\centering\includegraphics[width=0.5\textwidth, trim={20 20 25 25}, clip]{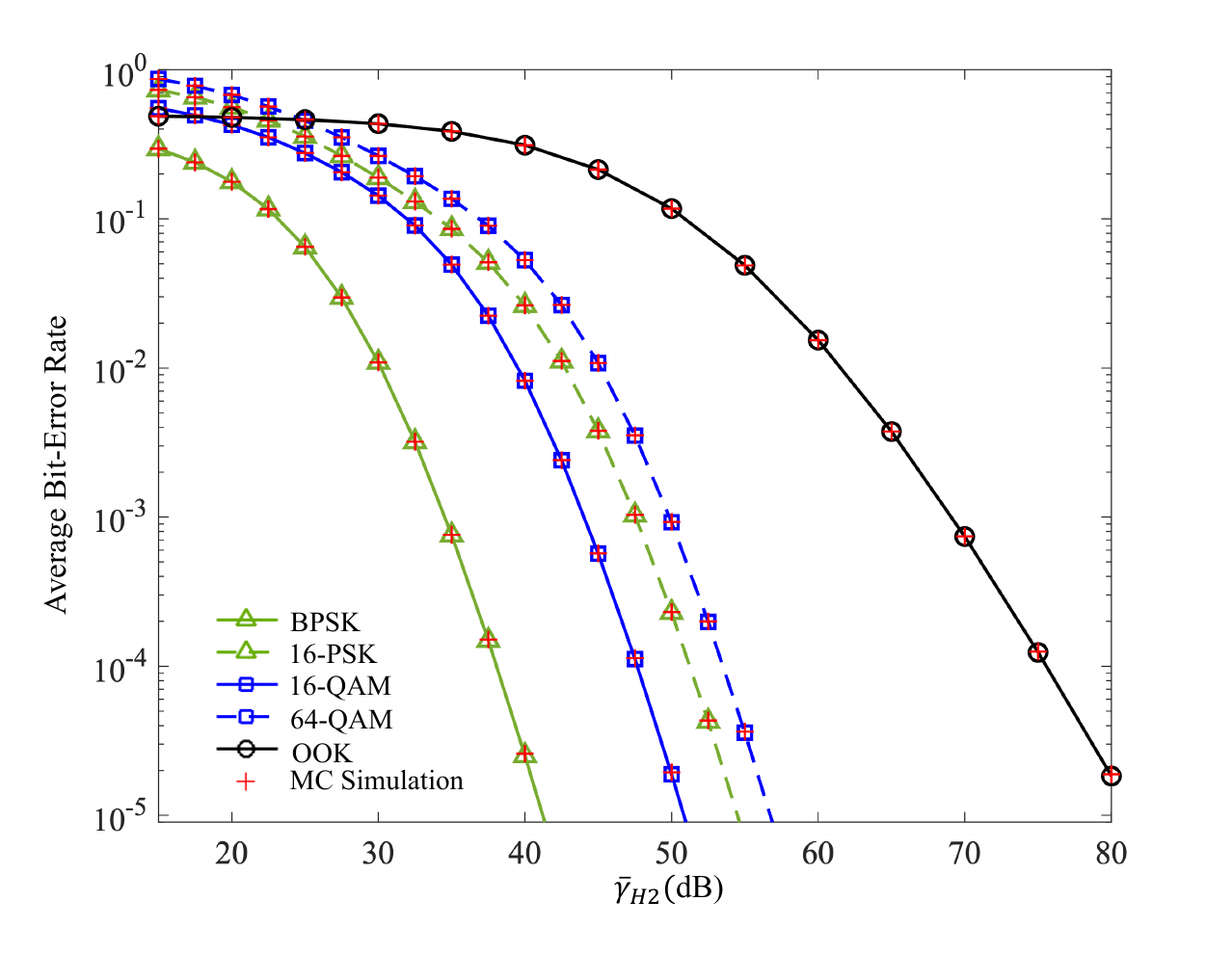}
    \caption{Average BER changes with system SNR for different modulation methods.}
    \label{ABER1}
\end{figure}
\begin{figure}[!ht]
\centering\includegraphics[width=0.5\textwidth, trim={10 20 25 25}, clip]{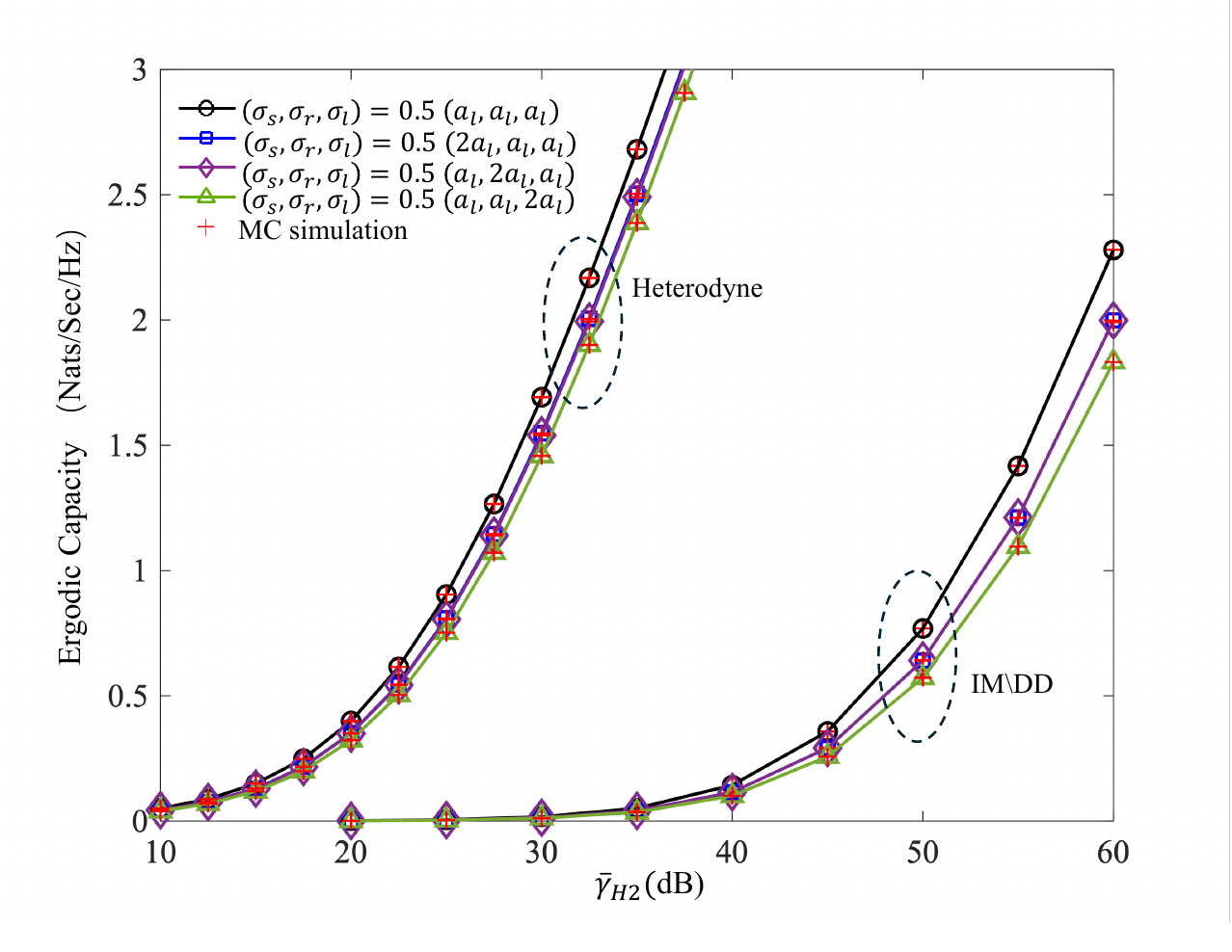}
    \caption{The trend of channel capacity as a function of SNR, under different fluctuation levels and detection methods.}
    \label{Capacity1}
\end{figure}

Finally, we analyze the system's  capacity under both detections, as shown in Fig. \ref{Capacity1}. The results reveal that heterodyne detection achieves significantly higher capacity than IM/DD. Using heterodyne detection as an example, at \(30 \, \mathrm{dB}\) SNR, the channel capacities for the four cases are \(1.69\, \mathrm{Nats/s/Hz}\), \(1.55\, \mathrm{Nats/s/Hz}\), \(1.54\, \mathrm{Nats/s/Hz}\), and \(1.46\, \mathrm{Nats/s/Hz}\), respectively, with laser source fluctuations having the least impact and combined fluctuations causing the most significant degradation. For IM/DD, the overall capacities are lower but follow the same trend. 

\subsection{OGS to User Link}
For the end-to-end system from the OGS, through the HAP and OIRS, to the User, we begin by analyzing the system's OP. Without loss of generality, we assume that the average SNR for both the OGS-to-HAP link and the HAP-to-User link via OIRS are identical. In simulations, we vary these two SNRs simultaneously (namely $\overline{\gamma}_{H1}=\overline{\gamma}_{H2}=\overline{\gamma}_H$, and $r_1=r_2=r$). To evaluate the impact of the OGS-to-HAP link on the overall system, we select the zenith angle of the emitted beam from the OGS, $\zeta_1$, as a variable, as the zenith angle determines the propagation distance of this link. Fig. \ref{OP2zeta} shows the variation in the end-to-end OP with system SNR under different zenith angles, \(\zeta_1\). 
\begin{figure}[!ht]
\centering\includegraphics[width=0.5\textwidth, trim={20 20 25 25}, clip]{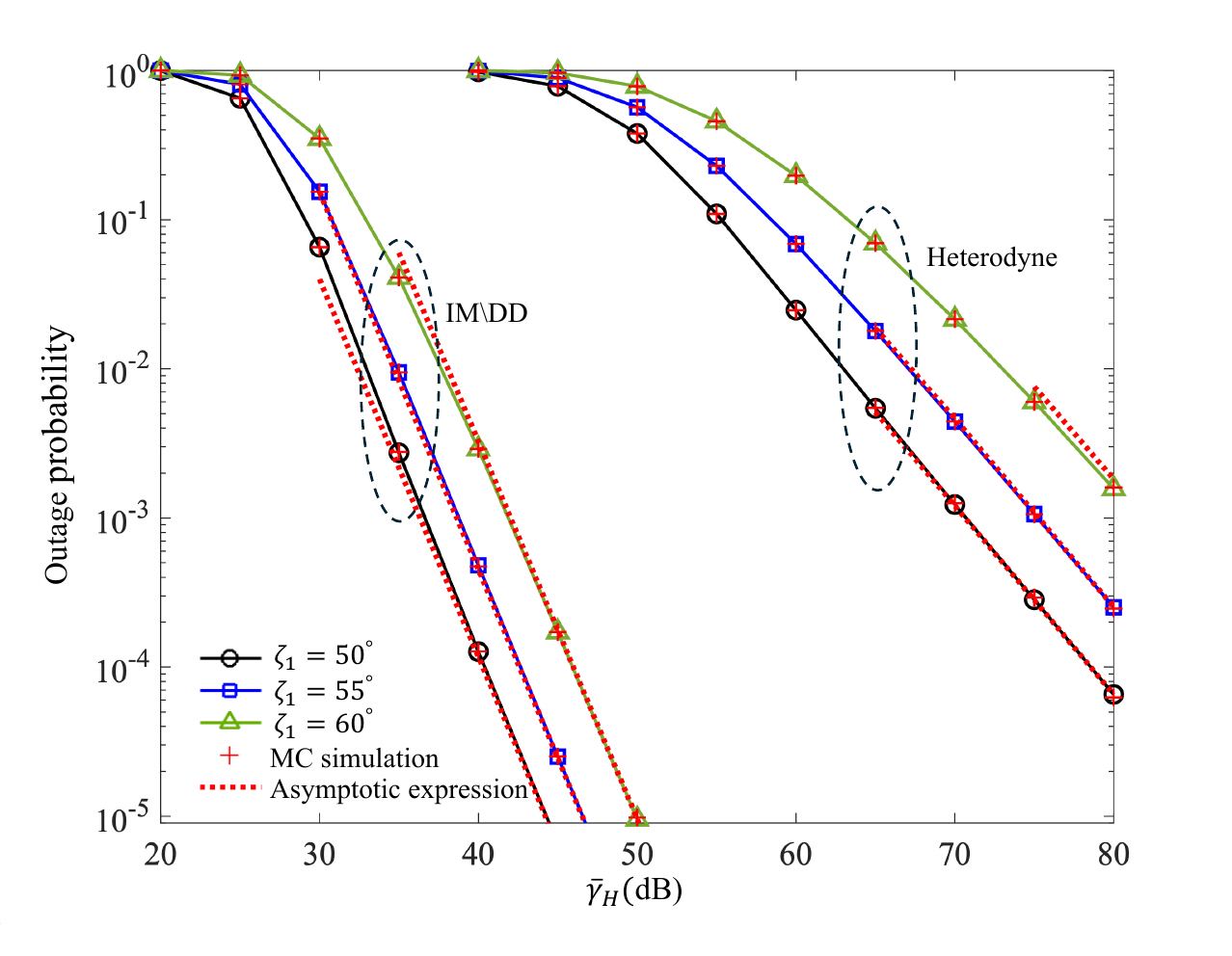}
    \caption{The end-to-end OP from the OGS, through the HAP and OIRS, to the User varies with the system SNR under different zenith angles \(\zeta_1\).}
    \label{OP2zeta}
\end{figure}
The system's OP increases as the zenith angle increases. A larger zenith angle results in a longer propagation distance from the OGS to the HAP, leading to more significant channel impairments from atmospheric turbulence and other factors. Specifically, for \( r = 1 \) and an SNR of 35 dB, the OP values corresponding to zenith angles of 50°, 55°, and 60° are \(2.7 \times 10^{-3}\), \(8.3 \times 10^{-3}\), and \(4.1 \times 10^{-2}\), respectively. The asymptotic results given by (\ref{CDFgammaA}) are very accurate and converge to the exact results at high SNR values.

To analyze the impact of the HAP-to-User link via OIRS on the overall system, we selected a representative level of fluctuation for this model. Given that the fluctuation of OIRS and receiving lens have similar effects on the system, we present results for only one. Fig.\ref{OP2sigma} illustrates how the end-to-end OP from the OGS, through the HAP and OIRS, to the User varies with the system SNR under different fluctuation levels.
\begin{figure}[!ht]
\centering\includegraphics[width=0.5\textwidth, trim={20 20 25 25}, clip]{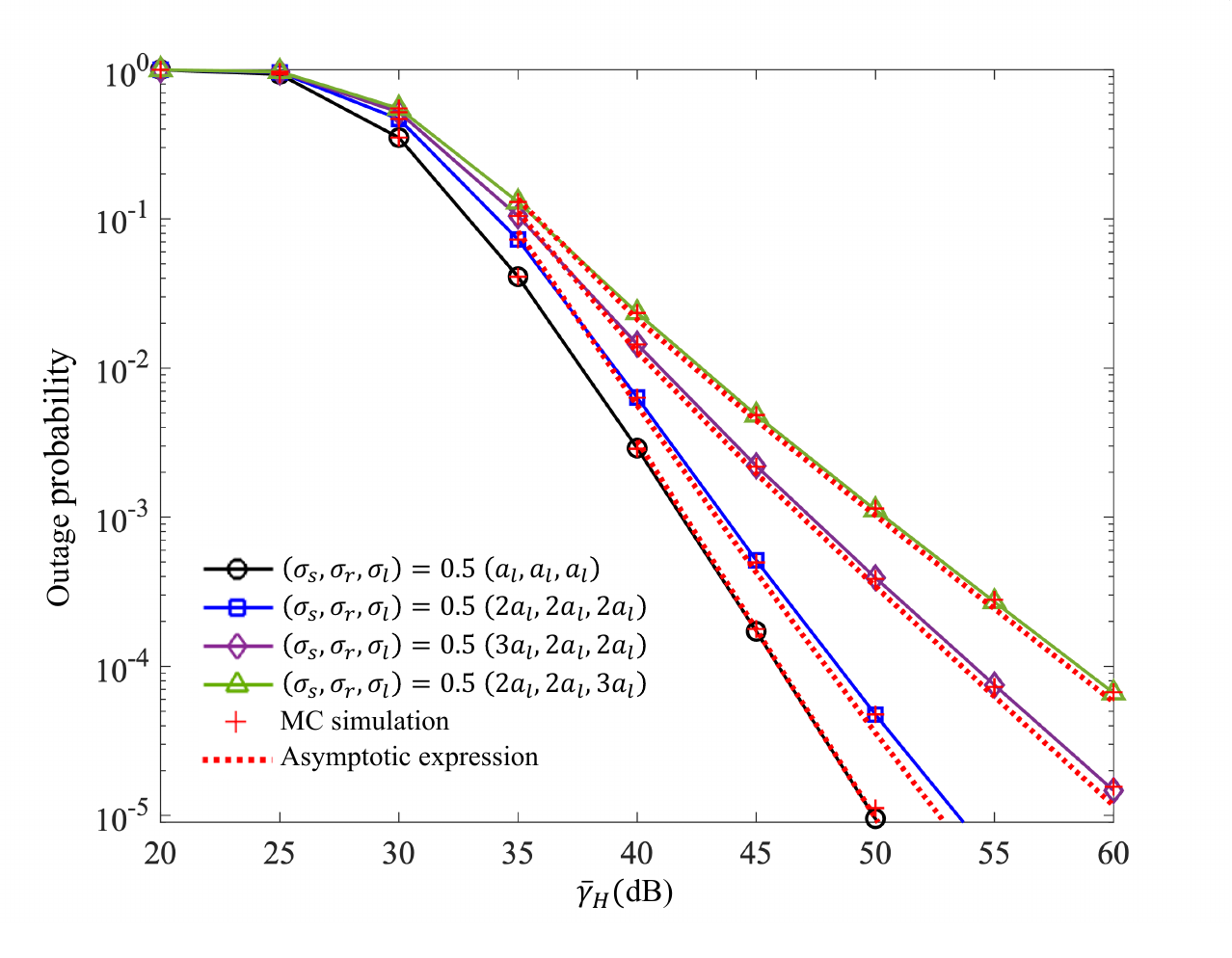}
    \caption{The end-to-end OP from the OGS, through the HAP and OIRS, to the User varies with the system SNR under different fluctuation levels.}
    \label{OP2sigma}
\end{figure}  Fluctuations in the receiving lens have a more significant impact on the system OP than those in the laser source. To clearly illustrate the impact of different fluctuations on system performance (namely $(\sigma_s, \sigma_r, \sigma_l) = 0.5 (2a_l, 2a_l, 2a_l)$), we first increased the fluctuation levels of the laser source, OIRS, and receiver lens equally (namely  $(\sigma_s, \sigma_r, \sigma_l) = 0.5 (3a_l, 2a_l, 2a_l)$,  $(\sigma_s, \sigma_r, \sigma_l) = 0.5 (2a_l, 2a_l, 3a_l)$). Based on this, we further increased the fluctuation levels of the laser source and receiver lens individually by the same amount. Specifically, for  SNR of 40 dB, the OP values corresponding to four different fluctuations cases are \(3.3 \times 10^{-3}\), \(5.6 \times 10^{-3}\), \(1.3 \times 10^{-2}\) and \(2.1 \times 10^{-2}\), respectively. The asymptotic results align perfectly with the derived analytical results in the high SNR range.

Next, we analyze how the system's average BER varies with the system SNR. Fig. \ref{ABER2} illustrates the end-to-end average BER from the OGS, through the HAP and OIRS, to the User under different modulation schemes. \begin{figure}[!ht]
\centering\includegraphics[scale=0.45, trim={20 20 25 25}, clip]{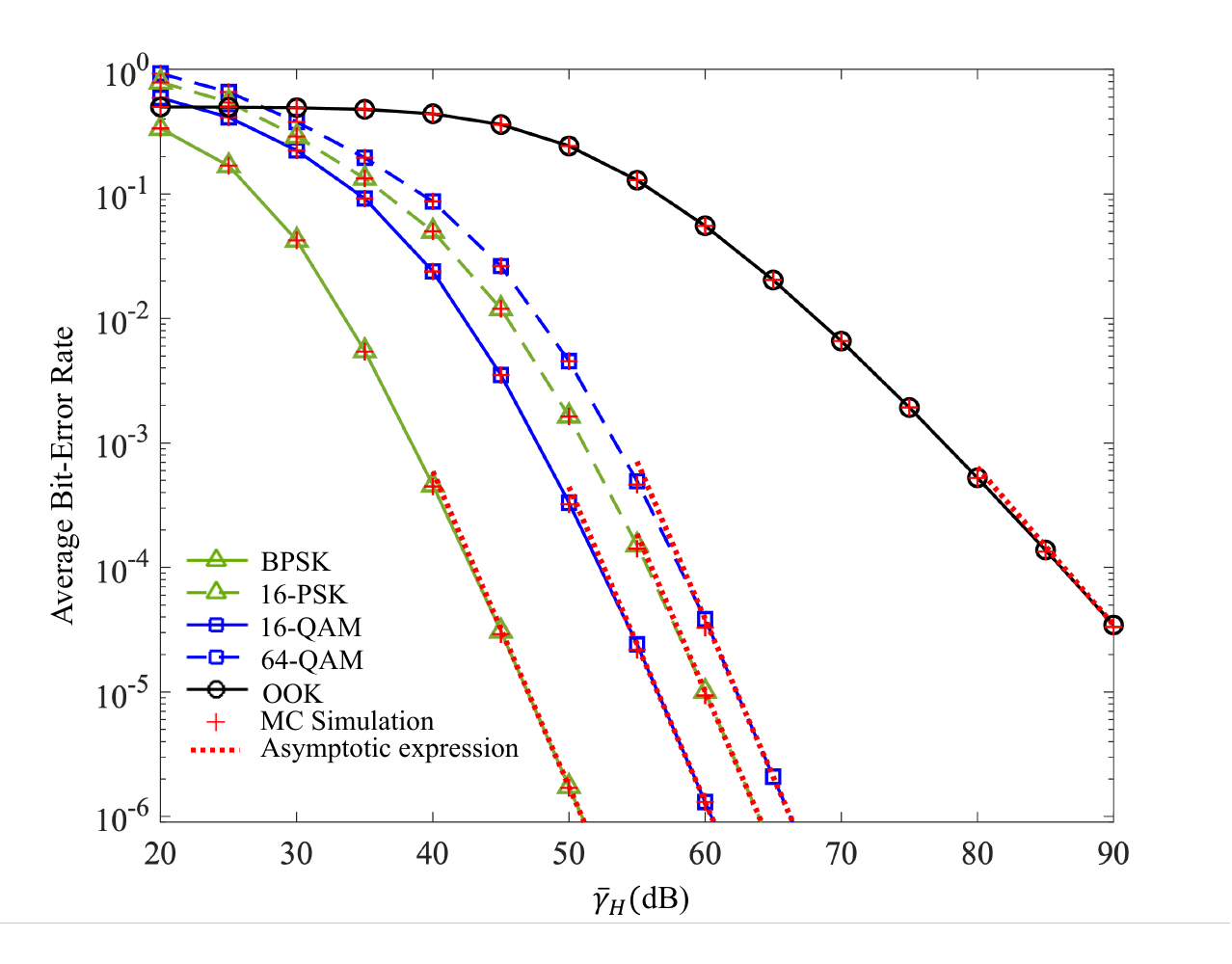}
    \caption{The end-to-end average BER from the OGS, through the HAP and OIRS, to the User varies with the system SNR under different modulation schemes.}
    \label{ABER2}
\end{figure} OOK exhibits the highest average BER, primarily due to its reliance on IM/DD detection, underscoring the performance benefits of heterodyne detection. For both M-QAM and M-PSK, BER increases with $M$ as higher $M$ improves spectral efficiency and data rate and raises error probability. The asymptotic results given in (\ref{IgammaA}) perfectly agree with the derived analytical results at the high SNR range.

\begin{figure}[!ht]
\centering\includegraphics[width=0.5\textwidth, trim={20 20 25 25}, clip]{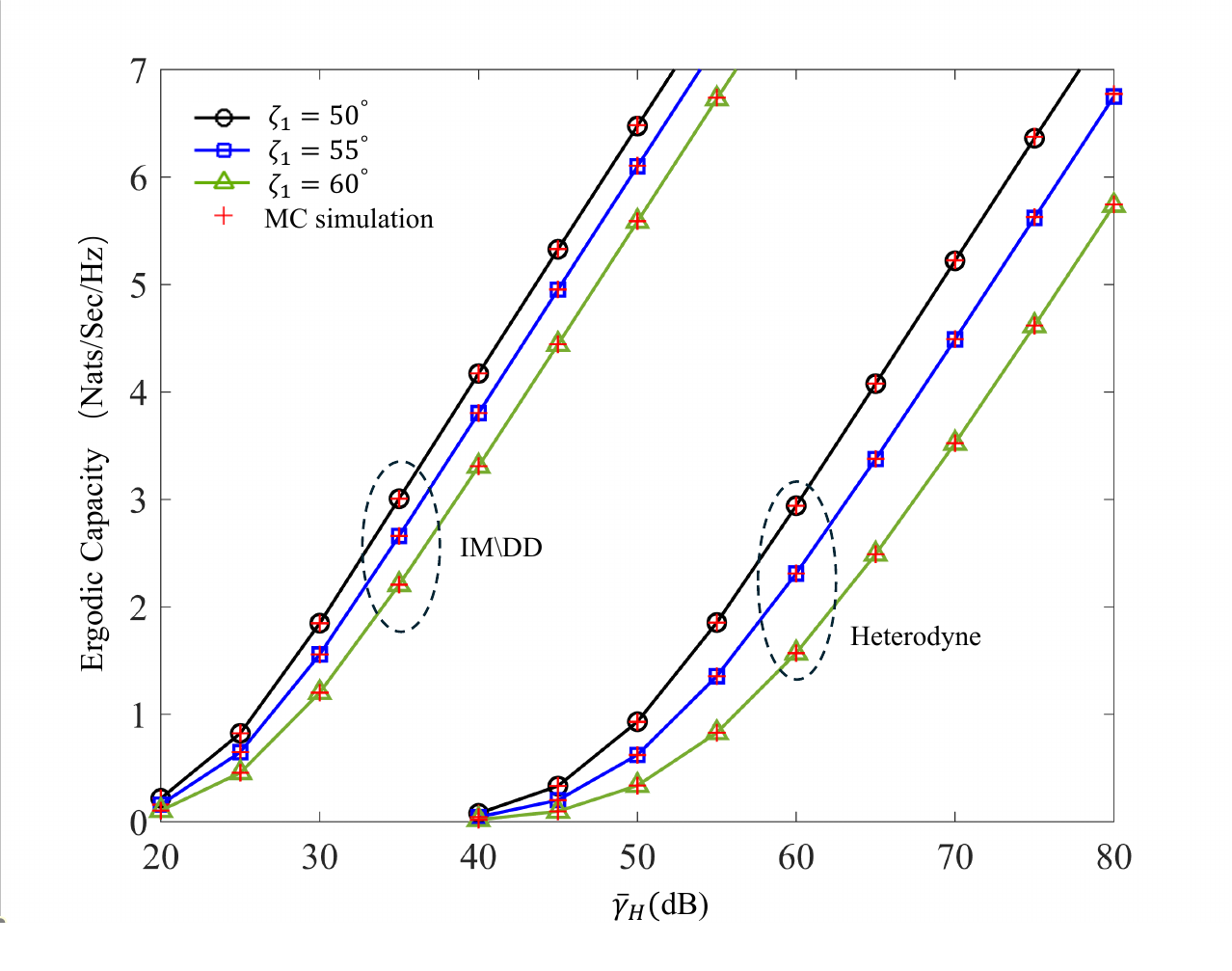}
    \caption{The end-to-end channel capacity from the OGS, through the HAP and OIRS, to the User varies with the system SNR.}
    \label{capacity2}
\end{figure} 
Finally, the end-to-end SNR's ergodic capacity for various zenith angle values is analyzed. Fig. \ref{capacity2} shows the end-to-end channel capacity from the OGS, through the HAP and OIRS, to the User as a function of average SNR. The channel capacity decreases as the zenith angle increases due to the longer propagation distance and more significant interference.  { Note that Fig. \ref{capacity2} only illustrates the performance for a representative zenith angle between the OGS and the HAP. While different zenith angles alter the slant range and thus the path loss, they affect the system performance in a similar fashion. Hence, showing one representative case sufficiently captures the trend without loss of generality.}

{ Fig.~\ref{compare} compares the OP performance between AF and DF relaying schemes under varying zenith angles $\zeta_1$. The OP for the DF scheme is evaluated based on the condition that communication is uninterrupted if and only if both link SNRs ($\gamma_{H1}$ and $\gamma_{H2}$) individually exceed the outage threshold $\gamma_{\text{th}}$, as described in \cite[Eq.~(26)]{singya2022mixed}. In contrast, for the AF scheme, communication continuity requires $\gamma > \gamma_{\text{th}}$. As illustrated, the DF scheme consistently achieves lower OP compared to AF, owing to its ability to independently decode and regenerate signals at the relay, thus effectively preventing noise amplification. However, it is notable from the numerical results that the AF scheme exhibits only slightly inferior performance compared to DF. For example, at $\bar{\gamma}_H = 30\,\text{dB}$ and $\zeta_1 = 60^{\circ}$, the OP of DF and AF schemes are approximately $0.26$ and $0.31$, respectively, reflecting only a modest difference. Thus, although the AF scheme incurs some performance degradation relative to DF, it offers substantial advantages in terms of lower complexity, reduced power consumption, and simpler hardware implementation, making AF relaying particularly suitable for deployment on energy-constrained platforms like HAPs.
\begin{figure}[!ht]
\centering
\includegraphics[scale=0.43]{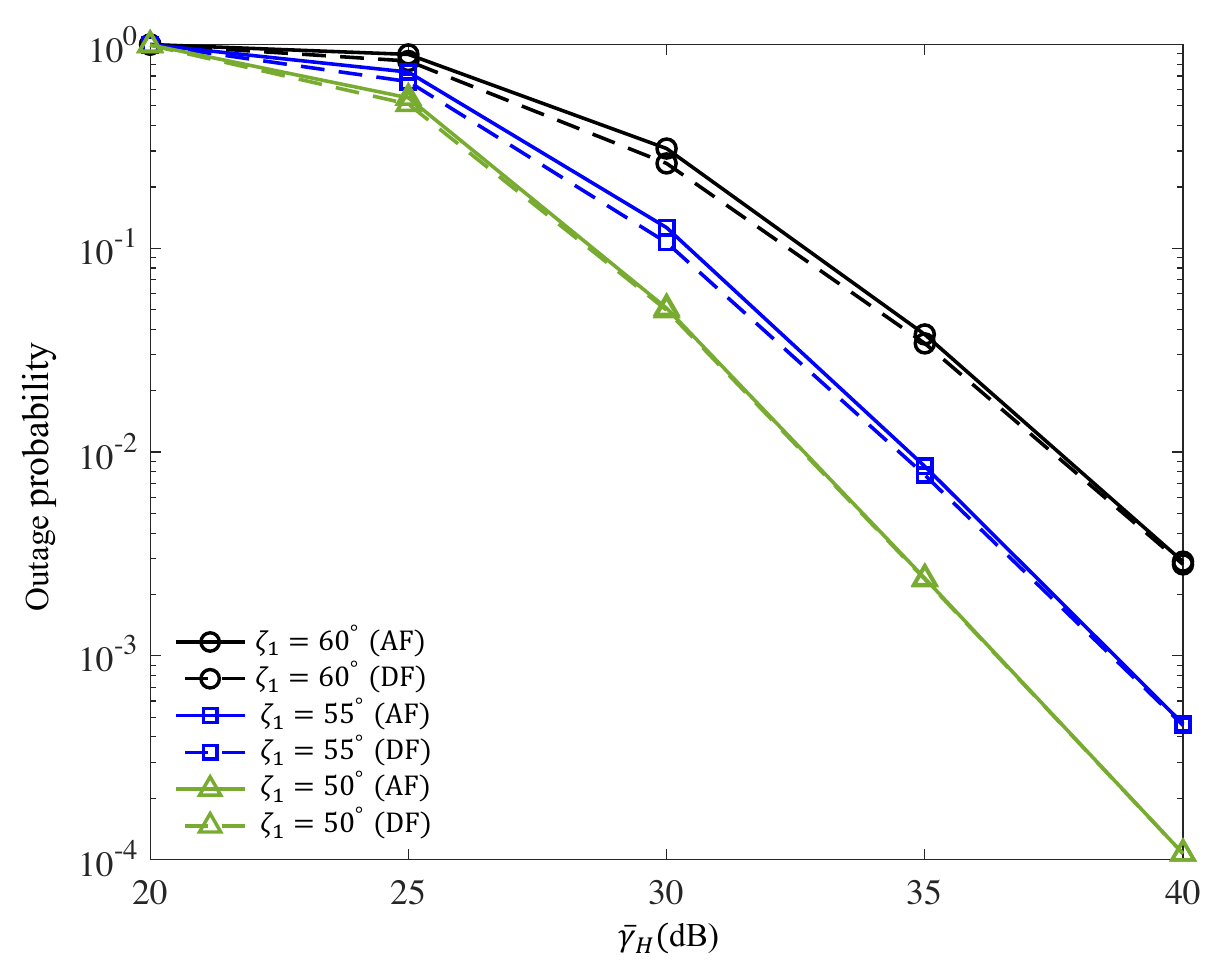}
\caption{Comparison of outage probability for AF and DF relaying schemes under different zenith angles $\zeta_1$.}
\label{compare}
\end{figure}}

\section{Conclusion}
In this work, we propose an innovative system that combines HAPs and OIRS to address LOS challenges in urban settings. Our three-hop system setup includes an OGS, HAP, OIRS, and a User.  For the OIRS link, we address key channel impairments such as atmospheric turbulence, pointing errors, attenuation, and GML by employing a GG model alongside a 3D GML model based on a Hoyt distribution. We  obtain closed-form expressions for OP and various performance metrics. The accuracy of our proposed approximation is verified against exact results, and the accuracy of our expressions for system OP and other performance metrics are confirmed through MC simulations.  Additionally, we analyze the end-to-end SNR statistics and derive closed-form expressions for OP and performance metrics for IM/DD and heterodyne detection methods. Asymptotic expressions are provided for high-SNR regimes, which facilitates the calculation of system diversity order.  At the same time, we also assess the impact of the OGS zenith angle and fluctuations in the laser source, OIRS, and receiving lens on overall system performance. { 
The outcomes of this research provide practical guidelines for system designers and engineers on selecting optimal OIRS deployment locations, setting appropriate beam alignment tolerances, and determining suitable relay amplification factors under various weather conditions. Furthermore, our performance comparison demonstrates that employing AF relaying at the HAP incurs only a minor performance loss compared to DF, while significantly reducing system complexity and power consumption, making AF a more practical solution for resource-constrained aerial platforms.
}

{
Although the proposed HAP and OIRS system offers significant advantages in coverage, several practical challenges remain. First, precise calibration and maintenance of the OIRS units, particularly ensuring accurate alignment under user mobility, require high resolution angle sensing and low latency closed loop control systems. Second, real time beam steering relies on energy efficient algorithms.
These aspects represent promising directions for future research to ensure the practical viability of the proposed architecture.
}

\appendices
\numberwithin{equation}{section}%
{ \section{PDF of $h_g2$}
\label{A0}
Using \cite[Eq.~(8.444)]{intetable}, the PDF expression of $h_{g2}$ in (\ref{PDFhg2}) can be re-written as
\begin{equation}
\begin{aligned}
    f_{h_{g2}}(h_{g2}) =&\,\, \frac{\varpi}{A_{02}} \left( \frac{h_{g2}}{A_{02}} \right) \left( \frac{(1+q_g^2)\varpi}{2q_g} \right)^{-1} \\&\times\sum_{k=0}^{\infty} \frac{1}{k! \Gamma(1+k)} \left( \frac{(1 - q_g^2)\varpi}{4q_g} \ln \left( \frac{h_{g2}}{A_{02}} \right) \right)^{2k}.
    \end{aligned}
\end{equation}
By replacing the infinity with \( N_k \), we can obtain an approximate version of (\ref{PDFhg2}). However, to ensure that the resulting approximation still represents a valid PDF, we need to include a normalization factor \( \mathcal{N} \) in the expression, then we can get (\ref{A2}). If (\ref{A2}) is still a valid PDF, then it should satisfy
\begin{equation}
\label{A3}
\begin{aligned}
      \sum_{k=0}^{N_k}& \frac{1}{k! \Gamma(1+k)}\int_0^{A_{02}} \left( \frac{(1 - q_g^2)\varpi}{4q_g} \ln \left( \frac{h_{g2}}{A_{02}} \right) \right)^{2k}\\& \times\frac{ \mathcal{N}\varpi}{A_{02}} \left( \frac{h_{g2}}{A_{02}} \right) \left( \frac{(1+q_g^2)\varpi}{2q_g} \right)^{-1} \mathrm{d}h_{g2}=1.
       \end{aligned}
\end{equation}
Using the integral formula \cite[Eq.~(4.272.6)]{intetable}  and (\ref{A3}), then the expression of $\mathcal{N}$ can be obtained as (\ref{mathcalN}).

\section{PDF of $h_2$}
\label{A}
The PDF of the channel gain $h_2$ can be written as
\begin{equation}
\label{A4}
    f_{h_2}(h_2) = \int_{\frac{h_2}  {A_{02} h_{p2}}}^{\infty} f_{h_{g2}} \left( \frac{h_2}{h_{p2} h_{a2}} \right)  \frac{f_{h_{a2}}(h_{a2})}{h_{p2} h_{a2}}\mathrm{d}h_{a2}.
\end{equation}

Let $t=\frac{h_2}{A_{02}h_{p2} h_{a2}}$, then (\ref{A4}) can be simplified into
\begin{equation}
\label{A5}
    f_{h_2}(h_2) = \int_0^1 \frac{f_{h_{g2}}(A_{02}t)}{t h_{p2}} f_{h_{a2}} \left( \frac{h_2}{A_{02} h_{p2} t} \right) \, \mathrm{d}h_{a2}.
\end{equation}
Substituting (\ref{PDFhg2}) and (\ref{pdfha2}) into (\ref{A5}), employing \cite[Eq.~(14)]{MeijerGalgorithm} and the definition of the Meijer-G function in \cite[Eq.~(9.301)]{intetable} yields
\begin{equation}
\begin{aligned}
&f_{h_2}(h_2) \,= \frac{2 \varpi \mathcal{N}}{h_2 \Gamma(\alpha_2)\Gamma(\beta_2)} \left( \frac{h_2 \alpha_2 \beta_2}{A_{02} h_{p2}} \right)^{\frac{\alpha_2 + \beta_2}{2}} \sum_{k=0}^{N_k} \frac{1}{k! \Gamma(1+k)} \\&\times\left( \frac{(1 - q_g^2)\varpi}{4q_g} \right)^{2k} 
\frac{1}{2 \pi \mathrm{i}} \int_{\mathcal{C}_1} \Gamma \left( \frac{\alpha_2 - \beta_2}{2} - s \right) \Gamma \left( \frac{\beta_2 - \alpha_2}{2} - s \right)\\& \times\left( \frac{\alpha_2 \beta_2 h_2}{A_{02} h_{p2}} \right)^s 
\int_0^1 t^{\frac{(1 + q_g^2) \varpi}{2 q_g} - 1 - \frac{\alpha_2 + \beta_2}{2} - s} (\ln(t))^{2k} \, \mathrm{d}t \mathrm{d}s.
\label{A6}
\end{aligned}
\end{equation}
where $\mathcal{C}_1$ and is the s-plane contours.  Using \cite[Eq.~(4.272.6)]{intetable}, we can obtain
\begin{equation}
\begin{aligned}
\int_0^1 &t^{\frac{(1 + q_g^2) \varpi}{2 q_g} - 1 - \frac{\alpha_2 + \beta_2}{2} - s} (\ln t)^{2k} \, \mathrm{d}t 
\\&= \frac{\Gamma(2k+1)}{\left( \frac{(1 + q_g^2) \varpi}{2 q_g} - \frac{\alpha_2 + \beta_2}{2} - s \right)^{2k}}
\end{aligned}
\label{A7}
\end{equation}
Substituting \eqref{A7} in to \eqref{A6} and using the definition of the Meijer-G function, (\ref{PDFh2}) can be obtained.}

\numberwithin{equation}{section}%
\section{$I_2(p_{B}, q_{B m})$}
\label{B}
Substituting (\ref{CDFH2})  into (\ref{DI2}) and using the Meijer-G function's primary definition in \cite[Eq.~(9.301)]{intetable}  yields
\begin{equation}
\begin{aligned}
I_2(&p_B, q_{Bm}) = \frac{1}{2} 
- \frac{q_{Bm}^{p_B}}{2 \Gamma(p_B)} \frac{\varpi \mathcal{N}}{\Gamma(\alpha_2) \Gamma(\beta_2)} \sum_{k=0}^{N_k} \frac{\Gamma(1 + 2k)}{k! \Gamma(1 + k)}\\
&\times \left[ \frac{(1 - q_g^2) \varpi}{4q_g} \right]^{2k}  \frac{1}{2\pi \mathrm{i}} \int_{\mathcal{C}_1}  \frac{\Gamma(-s)}{\Gamma(1 - s)} \Gamma(\alpha_2 - s) \Gamma(\beta_2 - s) \\
&\times \left[ \frac{\Gamma \left( \frac{(1 + q_g^2)\varpi}{2q_g} - s \right)}{\Gamma \left( 1 + \frac{(1 + q_g^2)\varpi}{2q_g} - s \right)} \right]^{1 + 2k}  \left[  \frac{\alpha_2 \beta_2}{A_{02} h_{p2}} \left( \frac{1}{\overline{\gamma}_{H2}} \right)^{\frac{1}{r_2}} \right]^s \\
&\times\int_0^\infty x^{p_B + \frac{s}{r_2} - 1} \exp(-q_{Bm} x) \mathrm{d}x  \mathrm{d}s.
\end{aligned}
\end{equation}
Using the definition of Gamma Function $\Gamma(z)$ in \cite{intetable}, $\Gamma(z)=\int_0^{\infty} t^{z-1} e^{-t} \mathrm{d} t$, (\ref{ABERIH2}) can be get.
\numberwithin{equation}{section}%
\section{The Capacity of $\gamma_{H2}$}
\label{C}
Substituting (\ref{PDFH2})  into (\ref{DefC2}) and using the definition of the MeijierG function  yields
{\begin{equation}
\scalebox{1}{$\begin{aligned}
&\overline{C}_2 = \frac{\varpi \mathcal{N}}{r_2 \Gamma(\alpha_2) \Gamma(\beta_2)} \sum_{k=0}^{N_k} \frac{\Gamma(1 + 2k)}{k! \Gamma(1 + k)} \left( \frac{(1 - q_g^2) \varpi}{4 q_g} \right)^{2k} \\&\times\frac{1}{2 \pi \mathrm{i}} \int_{\mathcal{C}_1}  \Gamma(\alpha_2 - s) \Gamma(\beta_2 - s) 
\left[ \frac{\Gamma \left( \frac{(1 + q_g^2) \varpi}{2 q_g} - s \right)}{\Gamma \left( 1 + \frac{(1 + q_g^2) \varpi}{2 q_g} - s \right)} \right]^{1 + 2k}\\&\times
 \left( \frac{\alpha_2 \beta_2}{A_{02} h_{p2}} \left( \frac{1}{\overline{\gamma}_{H2}} \right)^{\frac{1}{r_2}} \right)^s \int_0^\infty \gamma_{H2}^{\frac{s}{r_2} - 1} \ln(1 + c_0 \gamma_{H2})  \mathrm{d}\gamma_{H2} ds.
\end{aligned}$}
\end{equation}}
Using the integral identity $\int_0^{\infty} x^{\mu-1} \ln (1+\gamma x) \mathrm{d} x=\frac{\pi}{\mu \gamma^\mu \sin \mu \pi}$ in \cite[Eq.~(4.293/10)]{intetable} along with  $\Gamma(z) \Gamma(1-z)=\frac{\pi}{\sin \pi z}$ in \cite[Eq.(2), p99]{specialfunc},  the ergodic capacity can be derived in (\ref{C2}).

\numberwithin{equation}{section}%
\section{PDF and CDF of $\gamma$}
\label{E}
The CDF of the end-to end SNR $\gamma$ in (\ref{E2ESNR}) can be formulated as
\begin{align}
\label{E1}
    F_{\gamma}(\gamma)=\int_{0}^{\infty} F_{\gamma_{H1}}\left(\gamma\left(1+\frac{C}{x}\right)\right) f_{\gamma_{H2}}(x)\, \mathrm{d}x.
\end{align}
Substituting (\ref{CDFH}) and (\ref{PDFH2}) into (\ref{E1}), and using the definition of MeijerG function, we can get
\begin{equation}
\label{E2}
\scalebox{1}{$\begin{aligned}
&F_{\gamma}(\gamma) 
= 1 - \frac{\eta_s^2 \varpi \mathcal{N}}{   \Gamma(\alpha_1) \Gamma(\beta_1) r_2 \Gamma(\alpha_2) \Gamma(\beta_2)} 
\\
&\times\sum_{k=0}^{N_k} \frac{\Gamma(1 + 2k)}{k! \Gamma(1 + k)} \left( \frac{(1 - q_g^2) \varpi}{4 q_g} \right)^{2k} \\
&\times    \left(\frac{1}{2 \pi \mathrm{i}}\right)^2 
\int_{\mathcal{C}_1} \int_{\mathcal{C}_2}  \left[ \frac{\Gamma\left(\frac{(1 + q_g^2) \varpi}{2 q_g} - s\right)}{\Gamma\left(1 + \frac{(1 + q_g^2) \varpi}{2 q_g} - s\right)}\right]^{1 + 2k}  \\
&\times \frac{\Gamma(-t) \Gamma(\eta_s^2   - t) \Gamma(\alpha_1 - t) \Gamma(\beta_1 - t) \Gamma(\alpha_2 - s) \Gamma(\beta_2 - s)}
{\Gamma(1 + \eta_s^2   - t) \Gamma(1 - t)}\\&\times
\left( \frac{\alpha_1 \beta_1}{A_{01} h_{p_1}} \right)^t 
\left( \frac{\gamma}{\overline{\gamma}_{H1}} \right)^{\frac{t}{r_1}} 
\left( \frac{\alpha_2 \beta_2}{A_{02} h_{p_2}} \right)^s 
\left( \frac{1}{\overline{\gamma}_{H2}} \right)^{\frac{s}{r_2}} \\
& \times\int_{0}^{\infty} x^{\frac{s}{r_2} - 1} \left(1 + \frac{C}{x}\right)^{\frac{t}{r_1}} \, \mathrm{d}x \, \mathrm{d}t \, \mathrm{d}s \,  ,
\end{aligned}$}
\end{equation}
where ${\mathcal{C}_1}$  and ${\mathcal{C}_2}$ are the s-plane and the t-plane contours,
respectively. Then, by utilizing \cite[Eq. (3.251/11)]{intetable}, the CDF of $\gamma$ in (\ref{E2}) can be re-written in the form:
{\begin{equation}
\label{E4}
\scalebox{1}{$\begin{aligned}
F_{\gamma}(\gamma)&\, = 1 - \frac{\eta_s^2 \varpi \mathcal{N}}{   \Gamma(\alpha_1) \Gamma(\beta_1) r_1r_2 \Gamma(\alpha_2) \Gamma(\beta_2)} 
\\
&\times\sum_{k=0}^{N_k} \frac{\Gamma(1 + 2k)}{k! \Gamma(1 + k)} \left( \frac{(1 - q_g^2) \varpi}{4 q_g} \right)^{2k}    \\
&\times \left( \frac{1}{2 \pi \mathrm{i}} \right)^2 \int_{\mathcal{C}_1} \int_{\mathcal{C}_2} 
\left[ \frac{\Gamma\left(\frac{(1 + q_g^2) \varpi}{2 q_g} - s\right)}{\Gamma\left(1 + \frac{(1 + q_g^2) \varpi}{2 q_g} - s\right)} \right]^{1 + 2k} \\
& \times\frac{ \Gamma(\eta_s^2   - t) \Gamma(\alpha_1 - t) \Gamma(\beta_1 - t) \Gamma(\alpha_2 - s) \Gamma(\beta_2 - s)}
{\Gamma(1 + \eta_s^2   - t)  \Gamma\left(1-\frac{t}{r_1}\right)} \\
& \times\Gamma\left(-\frac{s}{r_2}\right) 
\left( \frac{\alpha_1 \beta_1}{A_{01} h_{p_1}} \right)^t 
\left( \frac{\gamma}{\overline{\gamma}_{H1}} \right)^{\frac{t}{r_1}} 
\left( \frac{\alpha_2 \beta_2}{A_{02} h_{p_2}} \right)^s  \\&\times
\left( \frac{C}{\overline{\gamma}_{H2}} \right)^{\frac{s}{r_2}}
\Gamma\left(-\frac{t}{r_1} + \frac{s}{r_2}\right) \, \mathrm{d}x \, \mathrm{d}t \, \mathrm{d}s \,   .
\end{aligned}$}
\end{equation}}
Then, applying (1.1) of \cite{Mittal}, (\ref{CDFgamma}) can be obtained.

Now, by differentiating (\ref{E4}) with respect to $\gamma$, we can compute the PDF of $\gamma$ as:
{\begin{equation}
\label{E5}
\scalebox{1}{$\begin{aligned}
f_{\gamma}(\gamma) \,&= \frac{\eta_s^2 \varpi \mathcal{N}}{   \Gamma(\alpha_1) \Gamma(\beta_1) r_1r_2 \Gamma(\alpha_2) \Gamma(\beta_2)} \\
&\times
\sum_{k=0}^{N_k} \frac{\Gamma(1 + 2k)}{k! \Gamma(1 + k)} \left( \frac{(1 - q_g^2) \varpi}{4 q_g} \right)^{2k}    \\
&\times \left( \frac{1}{2 \pi \mathrm{i}} \right)^2 \int_{\mathcal{C}_1} \int_{\mathcal{C}_2}  
\left[ \frac{\Gamma\left(\frac{(1 + q_g^2) \varpi}{2 q_g} - s\right)}{\Gamma\left(1 + \frac{(1 + q_g^2) \varpi}{2 q_g} - s\right)} \right]^{1 + 2k} \\
&\times \frac{ \Gamma(\eta_s^2   - t) \Gamma(\alpha_1 - t) \Gamma(\beta_1 - t) \Gamma(\alpha_2 - s) \Gamma(\beta_2 - s)}
{\Gamma(1 + \eta_s^2   - t)  \Gamma\left(-\frac{t}{r_1}\right)} \\
&\times \Gamma\left(-\frac{s}{r_2}\right) 
\left( \frac{\alpha_1 \beta_1}{A_{01} h_{p_1}} \right)^t 
\left( \frac{\gamma}{\overline{\gamma}_{H1}} \right)^{\frac{t}{r_1}} 
\left( \frac{\alpha_2 \beta_2}{A_{02} h_{p_2}} \right)^s \\&\times
\left( \frac{C}{\overline{\gamma}_{H2}} \right)^{\frac{s}{r_2}} 
\Gamma\left(-\frac{t}{r_1} + \frac{s}{r_2}\right) \, \mathrm{d}x \, \mathrm{d}t \, \mathrm{d}s \,   .
\end{aligned}$}
\end{equation}}
Then, applying (1.1) of \cite{Mittal}, (\ref{PDFgamma}) can be obtained.

\numberwithin{equation}{section}%
\section{Asymptotic result for  OP  of the end-to-end link}
\label{F}
According to 
the  Meijer-G function's primary definition in \cite[Eq.~(9.301)]{intetable}, the integral involving $t$ in (\ref{E4}) can be expressed in the form of a Meijer-G function. Then (\ref{E4}) can be reformulated as:
{\begin{equation}
\scalebox{1}{$\begin{aligned}
&F_{\gamma}(\gamma) = 1 - \frac{\eta_s^2 \varpi \mathcal{N}}{   \Gamma(\alpha_1) \Gamma(\beta_1) r_1 r_2 \Gamma(\alpha_2) \Gamma(\beta_2)} 
\sum_{k=0}^{N_k} \frac{\Gamma(1 + 2k)}{k! \Gamma(1 + k)}\\&\times \left( \frac{(1 - q_g^2) \varpi}{4 q_g} \right)^{2k}     \frac{1}{2 \pi \mathrm{i}} \int_{\mathcal{C}_1}\left[\frac{\Gamma\left(\frac{(1 + q_g^2) \varpi}{2 q_g} - s\right)}{\Gamma\left(1 + \frac{(1 + q_g^2) \varpi}{2 q_g} - s\right)} \right]^{1+2k}\\&\times{\rm H}^{4,0}_{2,4} \!\left[ \!\frac{\alpha_1 \beta_1}{A_{01} h_{p_1}} \left( \frac{\gamma}{\overline{\gamma}_{H1}} \right)^{\frac{1}{r_1}} 
\!\middle| \!\!\!
\begin{array}{c}
(1 + \eta_s^2  , 1), (1, \frac{1}{r_1}) \\
(\frac{s}{r_2}, \frac{1}{r_1}), (\eta_s^2  , 1), (\alpha_1, 1), (\beta_1, 1)
\end{array}\!\!\!\!\right] \\&\times \Gamma(\alpha_2 - s) \Gamma(\beta_2 - s) 
\Gamma\left(-\frac{s}{r_2}\right) 
 \left( \frac{\alpha_2 \beta_2}{A_{02} h_{p_2}} \right)^s \left( \frac{C}{\overline{\gamma}_{H2}} \right)^{\frac{s}{r_2}} \mathrm{d}s.   
\end{aligned}$}
\end{equation}}
For high values of $\overline{\gamma}_{H1}$,  the Fox-H function's can be approximated using \cite[Eq.~(1.8.4)]{Htran1}, then, for high values of $\overline{\gamma}_{H2}$,  the Fox-H function's can be approximated using \cite[Eq.~(1.8.4)]{Htran1}, after some simplification, (\ref{CDFgammaA}) can be obtained.

\numberwithin{equation}{section}%
\section{$I(p_B, q_{Bm})$}
\label{G}
The definition of $I(p_B, q_{Bm})$ is given as
\begin{equation}
\label{G1}
I(p_B, q_{Bm}) = \frac{q_{Bm}^{p_B}}{2 \Gamma(p_B)} \int_{0}^{\infty} \gamma^{p_B - 1} \exp(-q_{Bm} \gamma) F_{\gamma}(\gamma) \, \mathrm{d}\gamma.
\end{equation}
Substituting (\ref{E4}) in to (\ref{G1}), yields
{\begin{equation}
\scalebox{1}{$\begin{aligned}
&I(p_B, q_{Bm})= \frac{1}{2} 
- \frac{\eta_s^2 \varpi \mathcal{N} q_{Bm}^{p_B}}{2 \Gamma(\alpha_1) \Gamma(\beta_1) r_1 r_2 \Gamma(\alpha_2) \Gamma(\beta_2) \Gamma(p_B)} \\
&\times
\sum_{k=0}^{N_k} \frac{\Gamma(1 + 2k)}{k! \Gamma(1 + k)} \left(\! \frac{(1 - q_g^2) \varpi}{4 q_g}\! \right)^{2k}  \!\!\!  \!\!\\
&\times \left( \frac{1}{2 \pi \mathrm{i}} \right)^2 \!\!\!\int_{\mathcal{C}_1} \int_{\mathcal{C}_2}  \!\!\left[\!\!\frac{\Gamma\left(\frac{(1 + q_g^2) \varpi}{2 q_g} - s\right)}{\Gamma\!\!\left(1\! + \!\frac{(1 + q_g^2) \varpi}{2 q_g} \!-\! s\right)} \!\!\right]^{1 + 2k} \Gamma\left(\frac{s}{r_2}-\frac{t}{r_1} \right) \\&\times\frac{\!\Gamma(\eta_s^2   - t) \Gamma(\alpha_1 - t) \Gamma(\beta_1 - t) \Gamma(\alpha_2 - s) \Gamma(\beta_2 - s)} 
{\Gamma(1 + \eta_s^2   - t) \Gamma\left(1 - \frac{t}{r_1}\right)} \\
& \times
\left( \frac{\alpha_1 \beta_1}{A_{01} h_{p_1}} \right)^t 
\left( \frac{1}{\overline{\gamma}_{H1}} \right)^{\frac{t}{r_1}} 
\left( \frac{\alpha_2 \beta_2}{A_{02} h_{p_2}} \right)^s 
\left( \frac{C}{\overline{\gamma}_{H2}} \right)^{\frac{s}{r_2}}\\
&\times\int_{0}^{\infty} \gamma^{\frac{t}{r_1} + p_B - 1} \exp(-q_{Bm} \gamma) \, \mathrm{d}\gamma \mathrm{d}t\mathrm{d}s.
\end{aligned}$}
\end{equation}}

Using the definition of Gamma Function $\Gamma(z)$ in \cite{intetable}, $\Gamma(z)=\int_0^{\infty} t^{z-1} e^{-t} \mathrm{d} t$ and  (1.1) of \cite{Mittal}, (\ref{Igamma}) can be obtained.

Substituting (\ref{CDFgammaA}) in to (\ref{G1}) and using the definition of Gamma Function $\Gamma(z)$ in \cite{intetable}, $\Gamma(z)=\int_0^{\infty} t^{z-1} e^{-t} \mathrm{d} t$, (\ref{IgammaA}) can be obtained.

\numberwithin{equation}{section}%
\section{ergodic capacity of $\gamma$}
\label{H}
The ergodic capacity of the end-to-end system, where the FSO link operates under either heterodyne detection or IM/DD techniques, is given as \cite[Eq. (26)]{lapidoth}
\begin{align}\label{H1}
\overline{C}\triangleq \mathbb{E}[\ln(1+c_0\,\gamma)]=\int_{0}^{\infty}\ln(1+c_0\,\gamma)f_\gamma(\gamma)\,\mathrm{d}\gamma,
\end{align}
substituting (\ref{E5}) into (\ref{H1}) using the integral identity $\int_0^{\infty} x^{\mu-1} \ln (1+\gamma x) \mathrm{d} x=\frac{\pi}{\mu \gamma^\mu \sin \mu \pi}$ in \cite[Eq.~(4.293/10)]{intetable} along with  $\Gamma(z) \Gamma(1-z)=\frac{\pi}{\sin \pi z}$ in \cite[Eq.(2), p99]{specialfunc}, and applying (1.1) of \cite{Mittal}, the ergodic capacity can be derived in (\ref{Capacitygamma}).

\numberwithin{equation}{section}%
\section{The s-Moments of $\gamma$}
\label{I}
Using \cite[Eq.~(2.3)]{Mittal}, the PDF of $\gamma$ can be rewritten as
\begin{equation}
\label{I1}
\scalebox{0.95}{$\begin{aligned}
&f_{\gamma}(\gamma) = \frac{\eta_s^2 \varpi \mathcal{N}}{ \Gamma(\alpha_1) \Gamma(\beta_1) r_1 r_2 \Gamma(\alpha_2) \Gamma(\beta_2) \gamma} 
\sum_{k=0}^{N_k} \frac{\Gamma(1 + 2k)}{k! \, \Gamma(1 + k)}\\&\times \left( \frac{(1 - q_g^2) \varpi}{4 q_g} \right)^{2k} 
 \int_{0}^{\infty} x^{-1} \exp{(-x)}
{\rm H}_{2,3}^{3,0}\\&  \left[ \!\!\!\begin{array}{c}
\left(1 + \eta_s^2 \xi(\varphi), 1\right) \left(0, \frac{1}{r_1}\right) \\
\left(\eta_s^2 \xi(\varphi), 1\right) \left(\alpha_1, 1\right) \left(\beta_1, 1\right)
\end{array}\!\! \middle| \frac{\alpha_1 \beta_1}{A_{01} h_{p1}} \left( \frac{\gamma}{x \overline{\gamma}_{H1}} \right)^{\frac{1}{r_1}} \!\right] H_{2k+4,0}^{2k+1,2k+4}
\\&
\left[\!\!\!\begin{array}{c} {\left\{\left(\frac{(1 + q_g^2) \varpi}{2 q_g}+1,r_2\right)\right\}}_{2k+1}\\(\alpha_2, r_2) (\beta_2, r_2)(0,1){\left\{\!\!\left(\frac{(1 + q_g^2) \varpi}{2 q_g} ,r_2\right)\!\!\right\}}_{2k+1}\end{array}\middle| \begin{array}{cc}
\frac{\alpha_2 \beta_2}{A_{02} h_{p2}}\\\times \left( \frac{x C}{\overline{\gamma}_{H2}} \right)^{\frac{1}{r_2}}
\end{array} \!\!\right] \mathrm{d}x.
\end{aligned}$}
\end{equation}

The definition of $E(\gamma^s)$ is given as
\begin{equation}
\label{I2}
    E(\gamma^s) = \int_{0}^{\infty} \gamma^s f_{\gamma}(\gamma) \, \mathrm{d}\gamma.
\end{equation}

 Substituting \eqref{I1} into \eqref{I2} yields,
 \begin{equation}
\label{I3}
\scalebox{0.88}{$\begin{aligned}
& E(\gamma^s) = \frac{\eta_s^2 \varpi \mathcal{N}}{ \Gamma(\alpha_1) \Gamma(\beta_1) r_1 r_2 \Gamma(\alpha_2) \Gamma(\beta_2) }
\sum_{k=0}^{N_k} \frac{\Gamma(1 + 2k)}{k! \, \Gamma(1 + k)}\\& \times\left( \frac{(1 - q_g^2) \varpi}{4 q_g} \right)^{2k} 
 \int_{0}^{\infty} x^{-1} H_{0,1}^{1,0} \left[ (0,1) \middle| x \right] \int_{0}^{\infty}\gamma^{s-1}{\rm H}_{2,3}^{3,0}\\& \times
 \left[ \!\!\!\begin{array}{c}
\left(1 + \eta_s^2 \xi(\varphi), 1\right) \left(0, \frac{1}{r_1}\right) \\
\left(\eta_s^2 \xi(\varphi), 1\right) \left(\alpha_1, 1\right) \left(\beta_1, 1\right)
\end{array}\!\! \middle| \frac{\alpha_1 \beta_1}{A_{01} h_{p1}} \left( \frac{\gamma}{x \overline{\gamma}_{H1}} \right)^{\frac{1}{r_1}} \!\right] \mathrm{d} \gamma {\rm H}_{2k+4,0}^{2k+1,2k+4}
\\&\times\!
\left[\!\!\!\begin{array}{c} {\left\{\left(\frac{(1 + q_g^2) \varpi}{2 q_g}+1,r_2\right)\right\}}_{2k+1}\\(\alpha_2, r_2) (\beta_2, r_2)(0,1){\left\{\!\!\left(\frac{(1 + q_g^2) \varpi}{2 q_g} ,r_2\right)\!\!\right\}}_{2k+1}\end{array}\middle| \frac{\alpha_2 \beta_2}{A_{02} h_{p2}} \left( \frac{x C}{\overline{\gamma}_{H2}} \right)^{\frac{1}{r_2}} \!\!\right] \!\!\mathrm{d}x.
\end{aligned}$}
\end{equation}
Utilizing \cite[Eq.~(1.59)]{Htran2} and \cite[Eq.~(2.8.4)]{Htran1}, the moments of $\gamma$ can be computed according to (\ref{moment}).

\bibliographystyle{IEEEtran}
\bibliography{IEEEexample}

\end{document}